\documentclass[a4paper,11pt]{amsart}
\usepackage[ansinew]{inputenc}
 
\usepackage{geometry}

\usepackage{amsmath}
\usepackage{amsfonts}
\usepackage{amsthm}
\usepackage{amssymb}
\usepackage{graphicx}
\usepackage{tikz}
\usepackage{multirow}
\usepackage{enumerate}

\makeindex

\usepackage{soul} 
\sodef\so{}{.14em}{.4em plus.1em minus .1em}{.4em plus.1em minus .1em} 

\usepackage{aliascnt}
\usepackage{hyperref}
\hypersetup{ 
    colorlinks=false,
    pdfborder={0 0 0},}
\newcommand{\nref}[1]{\hyperref[#1]{\ref*{#1}}}

\usepackage{array} 
\usepackage[T1]{fontenc} 
\newcommand{\subtile}[1] 
{
	\vspace{-0.3cm}
	\begin{center}
 		{{\textsc{#1}}}\\
	\end{center}
	\vspace{0.1cm}
}

\newcommand{\imageplain}[2] 
{
      {\includegraphics[scale=#1]{#2.png}}
}
\newcommand{\image}[2] 
{
      \begin{center} 
	\includegraphics[scale=#1]{#2.png}
      \end{center}
}
\newcommand{\hamburger}[4] 
{
  \thispagestyle{empty}
  \vspace*{-2cm}
  \begin{flushright}
    ZMP-HH / #2 \\
    Hamburger Beitr{\"a}ge zur Mathematik Nr. #3 \\
    #4 \\
  \end{flushright}
  \vspace{0.5cm}
  \begin{center}
    \Large \bf
    #1
  \end{center}
  \vspace{0.5cm}
  \begin{center}	
    Simon Lentner \\
    Algebra and Number Theory, 
    University Hamburg,\\
    Bundesstra{\ss}e 55, D-20146 Hamburg \\
    \texttt{simon.lentner@uni-hamburg.de}
  \end{center}
  \vspace{0.5cm}

}

\makeatletter
\newcommand{\xleftrightarrow}[2][]{\ext@arrow 3359\leftrightarrowfill@{#1}{#2}}
\newcommand{\xdashrightarrow}[2][]{\ext@arrow 0359\rightarrowfill@@{#1}{#2}}
\newcommand{\xdashleftarrow}[2][]{\ext@arrow 3095\leftarrowfill@@{#1}{#2}}
\newcommand{\xdashleftrightarrow}[2][]{\ext@arrow
3359\leftrightarrowfill@@{#1}{#2}}
\def\rightarrowfill@@{\arrowfill@@\relax\relbar\rightarrow}
\def\leftarrowfill@@{\arrowfill@@\leftarrow\relbar\relax}
\def\leftrightarrowfill@@{\arrowfill@@\leftarrow\relbar\rightarrow}
\def\arrowfill@@#1#2#3#4{%
  $\m@th\thickmuskip0mu\medmuskip\thickmuskip\thinmuskip\thickmuskip
   \relax#4#1
   \xleaders\hbox{$#4#2$}\hfill
   #3$%
}
\makeatother

\newcommand{\Ksymb}[3]{      
      \begin{bmatrix}
      #1; #2 \\
      #3   
      \end{bmatrix}
}

\allowdisplaybreaks[1]

\newcommand{\ad}{\mbox{ad}}
\newcommand{\rank}{\mbox{rank}}
\renewcommand{\mod}{\;\mathrm{mod}\;}

\newcommand{\ord}{\mbox{ord}}

\newcommand{\U}{\mathcal{U}} 	
\newcommand{\B}{\mathcal{B}} 	
\newcommand{\g}{\mathfrak{g}} 	
\renewcommand{\sl}{\mathfrak{sl}}	
\newcommand{\Z}{\mathbb{Z}}  	
\newcommand{\N}{\mathbb{N}}  	
\newcommand{\C}{\mathbb{C}}  	
\renewcommand{\k}{\Bbbk}  	
\newcommand{\F}{\mathbb{F}}  	
\newcommand{\Q}{\mathbb{Q}}  	

\newcommand{\W}{\mathcal{W}} 

\renewcommand{\L}{\mathcal{L}} 
\newcommand{\K}{\mathcal{K}} 
\renewcommand{\l}{\ell} 

\newcommand{\ydm}{Yetter-Drinfel'd module }
\newcommand{\ydms}{Yetter-Drinfel'd modules }
\newcommand{\ydmP}{Yetter-Drinfel'd module}

\theoremstyle{plain}
\newtheorem{theorem}{Theorem}[section]
\newtheorem*{theoremX}{Theorem}

\newtheorem{corollary}[theorem]{Corollary}

\newtheorem{definition}[theorem]{Definition}

\newtheorem{example}[theorem]{Example}

\newtheorem{lemma}[theorem]{Lemma}

\newtheorem{problem}[theorem]{Problem}

\newtheorem{remark}[theorem]{Remark}

\begin{document}

\enlargethispage{10\baselineskip}
\hamburger{A Frobenius homomorphism for Lusztig's quantum groups for
arbitrary roots of unity}
{14-14}{516}{Oct 2014}
\thispagestyle{empty}

\begin{abstract}
For a finite dimensional semisimple Lie algebra and a root of unity, Lusztig
defined an infinite dimensional quantum group of divided powers. Under certain 
restrictions on the order of the root of unity, he constructed
a Frobenius homomorphism with finite dimensional Hopf kernel and
with image the universal enveloping algebra.

In this article we define and describe the Frobenius homomorphism for
arbitrary roots of unity by systematically using the theory of Nichols
algebras. In several new exceptional cases the Frobenius-Lusztig kernel is
associated to a different Lie algebra than the initial Lie algebra. Moreover,
the Frobenius homomorphism often switches short and long roots, and may produce  Lie algebras in a symmetrically braided category. 
\end{abstract}
  \makeatletter
  \@setabstract
  \makeatother

\tableofcontents
\vspace{-1cm}
{Partly supported by the DFG Priority Program 1388 ``Representation
theory''}

\newpage

\section{Introduction}

Fix a finite-dimensional semisimple Lie algebra $\g$ and a primitive $\ell$-th
root of unity $q$. For this data, Lusztig defined in 1989 an
infinite-dimensional complex Hopf algebra $U^\L_q(\g)$ called \emph{restricted
specialization} \cite{Lusz90a}\cite{Lusz90b}. He conjectured that for $\ell$
prime the representation theory of $U^\L_q(\g)$ is deeply connected to the one
of the respective affine Lie algebra as well as to the respective adjoint
Lie group over $\bar{\F}_\ell$. The former statement has been proven in a
certain form by Kazhdan and Lusztig in a series of papers, the latter statement
has been proven in 1994 by Andersen, Jantzen \& Soergel \cite{AJS94}.\\

For $\ell$ odd (and in case $\g=G_2$ not divisible by $3$) Lusztig had
in the cited papers obtained a remarkable Hopf algebra homomorphism to the
classical universal enveloping algebra $U(\g)$, which was for $\ell$ prime
related to the Frobenius homomorphism over the finite field $\F_\ell$. The Hopf
algebra kernel (more precisely the coinvariants) of this map turned out to be a
finite-dimensional Hopf algebra, called the \emph{small quantum group} or
\emph{Frobenius-Lusztig-kernel}, yielding an exact sequence of Hopf algebras: 
$$u_q^{\L}(\g)\xrightarrow{\;\subset\;}
U_q^\L(\g)\xrightarrow{\;Frob\;}U(\g)$$

The discovery of this finite-dimensional Hopf algebra $u_q^{\L}(\g)$ triggered
among others the development of the theory of finite-dimensional \emph{pointed
Hopf algebras} that culminated in the classification results by
Andruskiewitsch \& Schneider \cite{AS10,AnI11}, and a more general classification of possible quantum Borel
parts, the so-called \emph{Nichols algebras}, by Heckenberger \cite{Heck09} using
generalized root systems and Weyl groupoids. \\

The aim of this article is to consider a more general short exact sequence of
Hopf algebras without restrictions on the root of unity. Our approach
somewhat 
differs from Lusztig's explicit approach, as discussed below, and uses
crucially the theory of Nichols algebras. On the other hand we will restrict
ourselves in this article to the positive Borel part.\\

\enlargethispage{.5cm}
Our results are as follows: The cases
with $2,3\;|\;\ell$ exhibit in some cases a Frobenius-homomorphism to the Lie
algebra with the dual root system ($B_n\leftrightarrow C_n$) as has already been
observed in \cite{Lusz94}. As we shall see, moreover for small roots of unity
Lusztig's implicit definition of $u_q^\L(\g)$ inside $U_q^\L(\g)$ does {\bf not} coincide with the
common definition of $u_q(\g)$ by generators and relations. Altogether we shall
treat in this article arbitrary $q$ and find in all cases a Frobenius
homomorphism with finite-dimensional kernel
$$u_q(\g^{(0)})^+\cong u_q^\L(\g)^+\xrightarrow{\;\subset\;}
U_q^\L(\g)^+\xrightarrow{\;Frob\;}U(\g^{(\ell)})^+$$
with $\g^{(0)},\g,\g^{(\ell)}$ quite different Lie algebras, some in
braided symmetric tensor categories.\\

 An exotic case in our work is $\g=G_2,q=\pm i$. Here, both simple roots $\alpha_1,\alpha_2$ are present in  $\g^{(0)}$ and form a root system $A_2$ because of a premature Serre relation. However in total $\g^{(0)}=A_3$, generated by root vectors for $\alpha_1,\alpha_2,\alpha_{112}$ of $G_2$. This  seems to be rather exceptional and gives an interesting counterexample, for example to a conjecture\footnote{Thanks to I. Heckenberger for pointing this out} in \cite{An14}.   \\

The author would be very interested to understand the similar list for  \cite{Len14b} affine Lie algebras, where some situations seems still confusing, and to other Nichols algebras, where due to a lack of a deformation parameter it seems to be hard to write down a full quantum group of divided powers resp. a full quantum group with a large center.\\

As one application of our results, let us mention that \cite{FGST05}\cite{FT10}
have conjectured remarkable connections of $u_q(\g),U_q^\L(\g)$ to certain
vertex algebras. The degenerate case $\g=B_n,q=\pm i$ in the present article gives a Lusztig divided power algebra, which is an extension of a small quantum group of type $\g^{(0)}=A_1^{\times n}$ (the short roots) with a Lie algebra of type $\g^{(\ell)}=C_n$ (dual rootsystem). In the application, we have shown in \cite{FL17} that this case corresponds to the vertex algebra of $n$ pairs of symplectic fermions, with a global symmetry group $\mathrm{Sp}_{2n}$ of type $C_n$.\\

We now review the results of this article in more detail:\\

In Section \nref{sec_preliminaries} we fix
the Lie-theoretic notation and prove some technical preliminaries.
We also introduce Nichols algebras in the special cases relevant to this
article.\\
In Section \nref{sec_forms} we review the construction of the
Lusztig quantum group $U_q^\L(\g)$ via rational and integral forms and some
basic properties.\\
In Section \nref{sec_quantumgroup} we slightly improve some results in 
\cite{Lusz90a}\cite{Lusz90b} to account for arbitrary rational forms and
arbitrary roots of unity and start to target the algebra structure.\\
In Section \nref{sec_smalluq} we obtain the first main result: We use 
Nichols algebras to explicitly describe the assumed kernel $u_q^\L(\g)$
and its root system without restrictions on $q$:

\enlargethispage{1.5cm}
\begin{theoremX}[\nref{thm_smalluq}]
  For $\ord(q^2)>d_\alpha$ for all $\alpha\in \Phi^+$ we have
  $u_q^\L(\g,\Lambda)\cong  u_q(\g,\Lambda)$. If some 
  $\ord(q^2)\leq d_\alpha$  we can
  express $u_q^\L(\g,\Lambda)^+$ in terms of some ordinary
  $u_q(\g^{(0)},\Lambda)^+$ as
  follows:
  \begin{center}
  \begin{tabular}{ll|llll}
    $q\qquad$ & $\g\qquad$ & $u_q^\L(\g,\Lambda)^+\qquad$ & $\dim$ & primitive
    generators & Comment\\
    \hline
    $\pm 1$ & all & $\C$ & $1$ & none & trivial \\
    $\pm i$ & $B_n$ & $u_q(A_1^{\times n})^+$ & $2^n$ & $E_{\alpha_n},
      E_{\alpha_n+\alpha_{n-1}}, E_{\alpha_n+\alpha_{n-1}+\alpha_{n-2}}, \ldots$
      & short roots \\
    $\pm i$ & $C_n$ & $u_q(D_n)^+$ & $2^{n(n-1)}$ & $E_{\alpha_1},\ldots
      E_{\alpha_{n-1}},E_{\alpha_{n}+\alpha_{n-1}}$ & short roots \\
    $\pm i$ & $F_4$ & $u_q(D_4)^+$ & $2^{12}$ & $E_{\alpha_4},E_{\alpha_3},
    E_{\alpha_3+\alpha_2},E_{\alpha_3+\alpha_2+\alpha_1}$
      & short roots \\
    $\sqrt[3]{1},\sqrt[6]{1}$ & $G_2$ & $u_q(A_2)^+$ & $3^3$ & 
      $E_{\alpha_1},E_{\alpha_1+\alpha_2}$
      & short roots \\
    $\pm i$ & $G_2$ & $u_{\bar{q}}(A_3)^+$ & $2^6$ &
      $E_{\alpha_2},E_{\alpha_1},E_{2\alpha_1+\alpha_2}$ & exotic \\
  \end{tabular}
  \end{center}
\end{theoremX}\enlargethispage{1cm}

In Section \nref{sec_shortexactsequence} most of the work is done. The
strategy to obtain a Frobenius homomorphism is quite
conceptual, uses the previously obtained kernels and works for arbitrary
$q$:
\begin{enumerate}[a)]
  \item In Theorem \nref{thm_pairs} we extend a trick used by Lusztig
  in the simply-laced case: We prove that all pairs
  of roots can be simultaneously reflected into rank $2$ parabolic subsystems;
  we also add a complete classification of orbits. Hence it often suffices to
  verify statements only in rank $2$. 
  \item In Lemma \nref{lm_normal} we prove that $u_q^{\L,+}$ is a normal Hopf
  subalgebra of $U_q^{\L,+}$. This is done using the explicit description of
  $u_q^{\L,+}$ in the previous section together with the trick a). Our proof
  actually returns the adjoint action quite explicitly.
  \item We then consider abstractly the
  quotient $H$ of $U_q^{\L,+}$ by the normal Hopf subalgebra $u_q^{\L,+}$ (in
  the  category of $\Lambda$-Yetter-Drinfel'd modules). Using again trick a) we
  prove it is generated by primitive elements and has the expected commutator
  structure; it is hence isomorphic to some explicit $U(\g^{(\ell)})^+$. Note
  that the identification sometimes switches long and short root and
  picks up  additional factors. Except if certain lattices are even we prove $H$
  is an ordinary Hopf algebra, in the other cases $H$ is in a symmetrically
  braided category (so calling $\g^{(\ell)}$ a Lie algebras still makes sense). 
\end{enumerate}
Combining these results we finally achieve our main theorem:
\begin{theoremX}[\nref{thm_main}]
  Depending on $\g$ and $\ell$ we have the following exact sequences of Hopf
  algebras in the category of $\Lambda$-Yetter-Drinfel'd modules:
  $$u_q(\g^{(0)},\Lambda)^+\xrightarrow{\;\subset\;} U_q^\L(\g,
  \Lambda)^+\xrightarrow{\;Frob\;}U(\g^{(\ell)})^+$$
\begin{center}
\begin{tabular}{l|ll|l|ll}
& $\g\qquad$ & $\ell=\ord(q)$ & $\g^{(0)}\quad$ &
$\g^{(\ell)}\quad$ & is braided for\\
\hline\hline
\textnormal{Trivial cases:}
& all & $\ell=1$ & $0$ & $\g$ & no \\
& all & $\ell=2$ & $0$ & $\g$ & $ADE_{n\geq 2},C_{n\geq 3},F_4,G_2$ \\
\cline{2-6}
\multirow{4}{*}{\textnormal{Generic cases:}} 
& $ADE$ & $\ell\neq 1,2$ & $\g$ & $\g$ & $\ell=2\mod 4,n\geq 2$ \\
&$B_n$ & $4\nmid \ell\neq 1,2$ & $B_n$ & $B_n$ & no \\
&$C_n$ & $4\nmid \ell\neq 1,2$ & $C_n$ & $C_n$ & $\ell=2\mod 4,n\geq 3$ \\
&$F_4$ & $4\nmid \ell\neq 1,2$ & $F_4$ & $F_4$ & $\ell=2\mod 4$\\ 
&$G_2$ & $3\nmid\ell\neq 1,2,4$ & $G_2$ & $G_2$ & $\ell=2\mod 4$\\
\cline{2-6}
\multirow{8}{*}{\textnormal{Duality cases:}$\quad$}
& $B_n$ & $4|\ell\neq 4$ & $B_n$ & $C_n$ & $\ell=4\mod 8,n\geq 3$ \\
&& $\ell=4$ & $A_1^{\times n}$ & $C_n$ &  $n\geq 3$\\      
& $C_n$ & $4|\ell\neq 4$ & $C_n$ & $B_n$ &  no \\
&& $\ell=4$ & $D_n$ & $B_n$ & no \\
& $F_4$ & $4|\ell\neq 4$ & $F_4$ & $F_4$ & $\ell=4\mod 8$\\
&& $\ell=4$ & $D_4$ & $F_4$ & yes \\ 
& $G_2$ & $3|\ell\neq 3,6$ & $G_2$ & $G_2$ & $\ell=2\mod 4$\\
&& $\ell=3,6$ & $A_2$ & $G_2$ & $\ell=6$\\
\cline{2-6}
\textnormal{Exotic case:}
& $G_2$ & $\ell=4$ & $A_3$ & $G_2$ & no \\ 
\end{tabular}
\end{center}
\end{theoremX}
In the ``duality cases'' the Frobenius
homomorphism interchanges short and long roots. For small values of $\ell$ the
Frobenius-Lusztig kernel $u_q(\g^{(0)})$ usually degenerates, up to the point
where it vanishes in the ``trivial case'' $q=\pm 1$. Several cases are
``braided'', meaning $\g^{(\ell)}$ is a Lie algebras in a braided
symmetric category (precisely the even lattices
$\Lambda_R^{(\ell)}$ in Lemma \nref{lm_ellLattice}). The ``exotic case'' will
exhibit strange phenomena throughout this article.\\

\enlargethispage{1.5cm}
In Section \nref{sec_question} we state some open problems in the context of
this article.\\

\section{Preliminaries}\label{sec_preliminaries}

\subsection{Lie Theory}
Let $\g$ be a finite-dimensional, semisimple complex Lie algebra with simple
roots $\alpha_i$ indexed by $i\in I$ and a set of positive roots $\Phi^+$.
Denote the Killing form by $(,)$, normalized such that $(\alpha,\alpha)=2$ for
the short roots. The Cartan matrix is 
$$a_{ij}=2\frac{(\alpha_i,\alpha_j)}{(\alpha_i,\alpha_i)}$$
Be warned that there are different conventions for the index order of $a$, here
we use the convention usual in the theory of quantum groups.\\

It is custom to call $d_\alpha:=(\alpha,\alpha)/2$ with $d_\alpha\in\{1,2,3\}$,
especially $d_i:=d_{\alpha_i}$, which only depends on the orbit of $\alpha$
under the Weyl group. In this notation $(\alpha_i,\alpha_j)=d_ia_{ij}$. 

\begin{definition}
    The \emph{root lattice} $\Lambda_R=\Lambda_R(\g)$ is the free abelian
    group with $\rank(\Lambda_R)=\rank(\g)=|I|$ and is generated by
    $K_{\alpha_i}$ for each simple root $\alpha_i$. We denote general
    group elements in $\Lambda_R$ by $K_{\alpha}$ for elements $\alpha$ in
    the root lattice of $\g$. 
    The Killing form induces an integral pairing of abelian groups, turning
    $\Lambda_R$ into an \emph{integral lattice}: 
    $$(\_,\_):\;\Lambda_R\times \Lambda_R \to \Z$$
    $$( K_{\alpha},K_{\beta})
      := (\alpha,\beta)$$
\end{definition}
\begin{definition}    
    The \emph{weight lattice} $\Lambda_W=\Lambda_W(\g)$ 
    is the free abelian group with $\rank(\Lambda_W)=\rank(\g)$ generated by
    $K_{\lambda_i}$ for
    each fundamental dominant weight $\lambda_i$. We denote general
    group elements in $\Lambda_W$ by $K_{\lambda}$ with $\lambda$ in the
    weight lattice of $\g$.
    It is a standard fact of Lie theory (cf. \cite{Hum72}, Section 13.1) that
    the root lattice is contained in the weight lattice and we shall in what
    follows tacitly identify $\Lambda_R\subset \Lambda_W$. Moreover it is known
    that the pairing on $\Lambda_R$ can be extended to a \emph{integral
    pairing}:
    $$( \_,\_):\;\Lambda_W\times \Lambda_R \to \Z$$
    $$( K_{\lambda},K_{\beta} ):= (\lambda,\beta)$$
    Note that for multiply-laced $\g$ the group $\Lambda_W$ is \emph{no
    integral lattice}.
\end{definition}

For later use, we also define the following sublattice of the root
lattice $\Lambda_R$:
\begin{definition}
    The $\ell$-lattice $\Lambda_R^{(\ell)}\subset \Lambda_R$ for any
    positive integer $\ell$ is defined as follows
    $$\Lambda_R^{(\ell)}:=\langle K_{\alpha_i}^{\ell_i},\; i\in I\rangle$$
    where $\ell_i=\ell/gcd(\ell,2 d_i)$ is the order of $q^{2d_i}$ for $q$
    a primitive $\ell$-th root of unity. More generally we define for any root
    $\ell_\alpha=\ell/gcd(\ell,2d_\alpha)$, which only depends on the orbit of
    $\alpha$ under the Weyl group.
\end{definition}
\begin{example}
    In the case where $\g$ is simply-laced (hence all $d_\alpha=1$) we have
    $$
    \ell_i=\begin{cases}
      \ell, & \ell\;odd\\ 
      \frac{\ell}{2},&\ell\;even\\
    \end{cases}
    \qquad\Lambda_R^{(\ell)}=\begin{cases}
      \ell\cdot\Lambda_R, & \ell\;odd\\ 
      \frac{\ell}{2}\cdot\Lambda_R,&\ell\;even\\
    \end{cases}$$
\end{example}

Frequently, later statements can be simplified if all $\ell_i=\ell$, which is
equivalent to the ``generic case'' $2\nmid \ell$ (and $3\nmid \ell$ for
$\g=G_2$). Moreover for small $\ell$ the set of roots with
$\ell_\alpha=1$ will be important. For later use we prove

\begin{lemma}\label{lm_ellLattice}
    For all $\alpha,\beta\in \Lambda_R^{(\ell)}$ we have
    $$(\alpha,\alpha)\in \ell\Z 
    \qquad (\alpha,\beta)\in \frac{\ell}{2}\Z$$
    Moreover we have $(\alpha,\beta)\in\ell\Z$ except in the following
    cases: 
    \begin{center}
    \begin{tabular}{ll}
     $\g$ & Exceptions \\
     \hline
     $A_n,D_n,E_6,E_7,E_8,G_2$ & $\ell=2\mod 4$\\
     $B_n,n\geq 3$ & $\ell=4\mod 8$\\
     $C_n,n\geq 3$ & $\ell=2\mod 4$\\
     $F_4$ & $\ell=2,4,6\mod 8$\\
    \end{tabular}
    \end{center}
\end{lemma}
The exceptions will correspond to braided cases of the short exact sequence
in the Main Theorem \nref{thm_main}.
\begin{proof}
  It is sufficient to check the condition $\ell|(\alpha,\beta)$ on the lattice
  basis $\ell_i\alpha_i,i\in I$. We check for each $i,j$ whether the
  quotient $X$ is an integer:
  $$X:=\frac{(\ell_i\alpha_i,\ell_j\alpha_j)}{\ell}
  =\frac{\ell_i\ell_j(\alpha_i,\alpha_j)}{\ell}
  =\frac{\ell\cdot (\alpha_i,\alpha_j)}{gcd(\ell,2d_i)\cdot gcd(\ell,2d_j)}$$
  We start by checking the cases $i=j$ where we find indeed
  $X=\frac{\ell}{gcd(\ell,2d_i)}\cdot
  \frac{2d_i}{gcd(\ell,2d_i)}\in \Z$.
  To check the cases $i\neq j$ we can restrict ourselves to Lie algebras of
  rank $2$, where we check the claim case by case:
  \begin{itemize}
   \item For type $A_1\times A_1$ we have $(\alpha_i,\alpha_j)=0$.
   \item For type $A_2$ we have $d_i=d_j=d\in\{1,2\}$ and $(\alpha_i,\alpha_j)=-d$, hence
    $$X=\frac{\ell\cdot (-d)}{gcd(\ell,2d)\cdot gcd(\ell,2d)}$$
    If $2\nmid\ell$ we have $gcd(\ell,2d)=gcd(\ell,d)$ and hence
    $X=\frac{\ell}{gcd(\ell,d)}
    \cdot\frac{-d}{gcd(\ell,d)}\in \Z$.
    If $d|\ell$ but $2d\nmid \ell$ we have $gcd(\ell,2d)=d$ and hence
    $X=\frac{\ell}{d}\cdot\frac{-d}{d}\in \Z$.
    If $4d|\ell$ we have 
    $X=\frac{\ell}{2d}\cdot\frac{-d}{2d}=\frac{\ell}{4d}\cdot\frac{-d}{d}
    \in \Z$. 
    If however $2d|\ell$ but $4d\nmid \ell$ we have
    $X=\frac{\ell}{2d}\cdot\frac{-d}{2d}\in \Z+\frac{1}{2}$.
    \item For type $B_2$ we have $d_i=1,d_j=2$ and $(\alpha_i,\alpha_j)=-2$.
    Hence
    $$X=\frac{\ell\cdot (-2)}{gcd(\ell,2)\cdot gcd(\ell,4)}
		=\frac{-2}{gcd(\ell,2)}\cdot \frac{\ell}{gcd(\ell,4)}\in\Z$$
    \item For type $G_2$ we have $d_i=1,d_j=3$ and $(\alpha_i,\alpha_j)=-3$.
    Hence
    $$X=\frac{\ell\cdot (-3)}{gcd(\ell,2)\cdot gcd(\ell,6)}$$
    If $2\nmid \ell$ or $4|\ell$ we have as for $A_2$ that $X\in\Z$, while for
    $2|\ell,4\nmid \ell$ we have $X\in\Z+\frac{1}{2}$.
  \end{itemize}
  The assertion follows now from considering all pairs of simple roots $(i,j)$: 
  \begin{itemize}
   \item For $\g$ simply-laced, all $(i,j)$ are either $A_1\times A_1$ or $A_2$
    for short roots $d=1$. The exceptional cases are hence $\ell=2\mod 4$ 
    whenever an edge exists i.e. $n\geq 2$.
   \item For $\g=C_n$, all $(i,j)$ are either $A_1\times A_1$ or $A_2$
    for short roots $d=1$ or $B_2=C_2$. The exceptional cases are hence
    $\ell=2\mod 4$ for $n\geq 3$.
  \item For $\g=B_n$, all $(i,j)$ are either $A_1\times A_1$ or $A_2$
    for long roots $d=2$ or $B_2=C_2$. The exceptional cases are hence
    $\ell=4\mod 8$ for $n\geq 3$.
  \item For $\g=F_4$, all $(i,j)$ are either $A_1\times A_1$ or $A_2$
    for short roots $d=1$ or long roots $d=2$ or $B_2=C_2$. The exceptional
    cases are hence $\ell=2\mod 4$ as well as $\ell=4\mod 8$.
  \item For $G_2$ we already calculated the exceptional cases to be $\ell=2\mod
    4$.
  \end{itemize}
\end{proof}

\subsection{Nichols algebras}\label{ssec_NicholsAlgebras}

Nichols algebras generalize the Borel parts of quantum groups in the
classification of pointed Hopf algebras, see \cite{AS10} Sec. 5.1. In this
article we only use the Nichols algebras appearing in ordinary quantum groups,
as briefly introduced in the following, but their use makes the
later constructions more transparent. For a detailed account on Nichols
algebras see e.g. \cite{HLecture08}.\\

\begin{definition} Assume we are over the base field $\C$. 
  A \emph{\ydmP} $M$ over a finite abelian group $\Gamma$
  is a $\Gamma$-graded vector space, $M=\bigoplus_{g\in \Gamma} M_g$
  with a $\Gamma$-action on $M$ such that $g.M_h=M_{h}$. 
\end{definition}

The category of \ydms form a braided category. Let $M$ be an $n$-dimensional
\ydm over the field $\C$, then we may choose a homogeneous vector space basis
$v_i$ with grading some $g_i\in\Gamma$ and express the action via
$g_i.v_j=q_{ij}v_j$ for some $q_{ij}\in\C^\times$. Then the braiding of $M$ 
has the form 
$$v_i\otimes v_j \mapsto q_{ij}\;v_j\otimes v_i$$

\begin{definition}
  Consider the tensor algebra $T(M)$, which
  can be identified with the algebra of words in the letters
  $v_i$ and is again a $\Gamma$-\ydmP. We can uniquely obtain
  \emph{skew derivations} $\partial_i:\;T(M)\rightarrow T(M)$ by
  $$\partial_k(1)=0
  \qquad\partial_k(v_l)=\delta_{kl}1
  \qquad\partial_k(x\cdot y)=\partial_k(x)\cdot (g_k.y)+x\cdot \partial_k(y)$$
  The \emph{Nichols algebra} $\B(M)$ is the
  quotient of $T(M)$ by the largest
  homogeneous ideal $\mathfrak{I}$ in degree $\geq 2$, invariant under
  all $\partial_k$. It is a Hopf algebra in the braided category of
  $\Gamma$-\ydmP. 
\end{definition}

Heckenberger classified all finite-dimensional Nichols algebras over
finite abelian groups $\Gamma$ in \cite{Heck09}. In the present article, we
only need the following examples:\\

Let $\Phi^+$ be the set of positive roots for a finite-dimensional complex
semisimple Lie algebra $\g$ of rank $n$ and normalized Killing form $(,)$ as in
the Lie theory preliminaries. Let $q$ be a primitive $\ell$-th root of unity.
Then the \ydm defined by the braiding matrix $q_{ij}:=q^{(\alpha_i,\alpha_j)}$
has a finite-dimensional Nichols algebra $\B(M)$ iff all
$q_{ii}=q^{d_{\alpha_i}}$ are $\neq 1$. More precisely, $\B(M)$ has a PBW-like
basis associated to $\Phi$ and especially the dimension is
$$\dim(\B(M))=\prod_{\alpha\in \Phi^+} \ord(q^{(\alpha,\alpha)})
=\prod_{\alpha\in \Phi^+} \ord(q^{d_\alpha})$$ 
\emph{unless} the case $\g=G_2,\ell=4$, which is excluded in
Heckenberger's list entry for $G_2$ (\cite{Heck06} Figure 1 Row 11). Indeed,
the braiding matrix is in this case equal to the braiding matrix for
$A_2$, namely
$q_{ij}=\begin{pmatrix}
   -1 & \sqrt{-1} \\
   \sqrt{-1} & -1 
 \end{pmatrix}$ and $\dim(\B(M))=2\cdot 2\cdot 2$ is less than expected for
$G_2$.\\

The condition $q_{ii}\neq 1$ and the exceptional case
$\g=G_2,\ell=4$ will play a prominent role in the present article.
It will be the direct cause why the Borel part of the small quantum group
$u_q(\g)^+$ is for small $\ell$ \emph{not} isomorphic to the corresponding
Nichols algebra $\B(M)$ as one might expect.

\subsection{Coradical extensions}\label{sec_ExtensionOfScalars}

We introduce the following tool without referring to quantum groups. It will
later allow us to quickly transport results about the adjoint
rational form ($\Lambda=\Lambda_R$) in literature to arbitrary $\Lambda$.\\

  Suppose $H$ a Hopf algebra over a commutative ring $\k$ with group of
  grouplikes $G(H)$ and fix some subgroup the group of 
  grouplikes $\Lambda_R \subset G(H)$.
  Let $\Lambda\rhd \Lambda_R$ be a group containing $\Lambda_R$ normally and
  let $\rho:\k[\Lambda]\otimes H\to H$ be an action, such that
\begin{itemize}
 \item The action $\rho$ turns $H$ into a $\k[\Lambda]$-module Hopf algebra.
 \item The action $\rho$ restricts on $\k[\Lambda_R]\subset \k[\Lambda]$ to the
  adjoint representation $\rho_R$  of the Hopf subalgebra $\k[\Lambda_R]\subset
  H$.
 \item The action $\rho$ restricts on $\k[\Lambda_R]\subset H$ to the
  adjoint representation $\rho_\Lambda$ of $\k[\Lambda]$ on the Hopf subalgebra
  $\k[\Lambda_R]$, given by conjugacy action of the group $\Lambda$ on the
normal  subgroup $\Lambda_R$.
\end{itemize}

\begin{theorem}\label{thm_ExtensionOfScalars}
  The Hopf algebra structure
  on the smash-product $\k[\Lambda]\ltimes H$ factorizes to a
  Hopf algebra structure on the vector space $\k[\Lambda]\otimes_{\k[\Lambda_R]}
  H$
  $$H_{\Lambda}:=\k[\Lambda]\ltimes_{\k[\Lambda_R]} H$$
  where the left-/right $\k[\Lambda_R]$-module structures are the
  multiplication with respect to the inclusions 
  $\k[\Lambda_R]\subset \k[\Lambda]$ and $\k[\Lambda_R]\subset G(H)\subset H$
\end{theorem}
  Especially the choice $\Lambda=\Lambda_R$ recovers 
  $H_{\Lambda_R}:=\k[\Lambda_R]\otimes_{\k[\Lambda_R]} H=H$.
\begin{proof}
  The smash-product of two Hopf algebras $H'_{\Lambda}:=\k[\Lambda]\ltimes H$
  with respect to an action $\rho$ on the Hopf algebra $H$ is the vector spaces
  $H'_\Lambda:=\k[\Lambda]\otimes_{\k} H$
  with the coalgebra structure of the tensor product and
  the multiplication $\mu_{H'_\Lambda}$ given for
  $g,h\in \Lambda, x,y\in H$ by:
  $$\mu_{H'_{\Lambda}}\left((g\otimes x)\otimes (h\otimes y)\right)
  =gh^{(1)}\otimes (\rho(S(h^{(2)})\otimes x)\cdot y)
  =gh\otimes (\rho(h^{-1}\otimes x)\cdot y)$$
  We have to show that the structures $1_{H'_\lambda},\mu_{H'_\lambda},
  \Delta_{H'_\lambda},\epsilon_{H'_\lambda}$  factorize  over the
  surjection
  $$\phi:\;H_\Lambda':=\k[\Lambda]\otimes_{\k} H\longrightarrow
  \k[\Lambda]\otimes_{\k[\Lambda_R]} H=:H_\Lambda$$
  \begin{itemize}
   \item The multiplication $\mu_{H'_\lambda}$ factorizes as follows: For all
    $g,h\in \Lambda,t\in\Lambda_R,x,y\in H$ we have 
    \begin{align*}
      (\phi\circ\mu_{H'_\lambda})((gt\otimes x)\cdot (h\otimes y))
      &=gth\otimes_{\k[\Lambda_R]} \rho(h^{-1}\otimes x)y\\
      &=gh\otimes_{\k[\Lambda_R]}\rho(h^{-1}th\otimes \rho(h^{-1}\otimes x)) (h^{-1}th)y\\
      &=gh\otimes_{\k[\Lambda_R]}\rho(h^{-1}\otimes\rho_R(t\otimes x)) \rho(h^{-1}\otimes t)y\\   
      &=gh\otimes_{\k[\Lambda_R]} \rho(h^{-1}\otimes tx)y\\
      &=(\phi\circ\mu_{H'_\lambda})((g\otimes tx)\cdot (h\otimes y))
    \end{align*}
    On the other hand we have
    \begin{align*}
      (\phi\circ\mu_{H'_\lambda})((g\otimes x)\cdot (ht\otimes y))
      &=ght\otimes_{\k[\Lambda_R]} \rho((ht)^{-1}\otimes x)y\\
      &=ght\otimes_{\k[\Lambda_R]} t^{-1}\rho(h^{-1}\otimes x)ty\\
      &=gh\otimes_{\k[\Lambda_R]} \rho(h^{-1}\otimes x)ty\\  
      &=(\phi\circ\mu_{H'_\lambda})((g\otimes x)\cdot (h\otimes ty))
    \end{align*}

   \item The unit $1_{H'_\lambda}$ maps to $\phi(1_{H'_\lambda})\in H_\Lambda$.
    \item The comultiplication $\Delta_{H'_\lambda}$ factorizes as follows: For
    all $g,h\in \Lambda,t\in\Lambda_R,x,y\in H$ we have $t$ grouplike and hence
    \begin{align*}
      (\phi\circ\Delta_{H'_\lambda})(gt\otimes x)
      &=\left(gt\otimes_{\k[\Lambda_R]} x^{(1)}\right)
      \otimes \left(gt\otimes_{\k[\Lambda_R]} x^{(2)}\right)\\
      &=\left(g\otimes_{\k[\Lambda_R]} tx^{(1)}\right)
      \otimes \left(g\otimes_{\k[\Lambda_R]} tx^{(2)}\right)\\
      &=\left(g\otimes_{\k[\Lambda_R]} (tx)^{(1)}\right)
      \otimes \left(g\otimes_{\k[\Lambda_R]} (tx)^{(2)}\right)\\
      &= (\phi\circ\Delta_{H'_\lambda})(g\otimes tx)
    \end{align*}
   \item The counit $\epsilon_{H'_\lambda}$ factorizes as follows: For all
    $g\in \Lambda,x\in H$ we have 
    \begin{align*}
      (\phi\circ\epsilon_{H'_\lambda})(gt\otimes x)
      &=\epsilon_{H'_\lambda}(gt)\cdot\epsilon_{H'_\lambda}(x)\\
      &=\epsilon_{H'_\lambda}(g)\cdot\epsilon_{H'_\lambda}(tx)\\
      &=(\phi\circ\epsilon_{H'_\lambda})(gt\otimes x)
   \end{align*}
  \end{itemize}
\end{proof}

\section{Different forms of quantum groups}\label{sec_forms}

We recall several Hopf algebras associated to $\g$ over various commutative
rings $\k$.

\begin{remark}
    The following notion is added for completeness and not used in the
    sequel: There is a so-called \emph{topological Hopf algebra}
    $U^{\C[[q]]}_q(\g)$ over the ring of formal power series $\k=\C[[q]]$
    cf. \cite{CP95} 6.5.1. It was defined by Drinfel'd (1987) and
    Jimbo (1985).
\end{remark}

\subsection{The rational forms}

We next define the \emph{rational form} $U^{\Q(q)}_q(\g)$. There are in fact
several rational forms $U^{\Q(q)}_q(\g,\Lambda)$ associated to the
$U^{\C[[q]]}_q(\g)$ that differ by a
choice of a subgroup $\Lambda_R\subset \Lambda\subset \Lambda_W$ resp. a choice
of a subgroup in the \emph{fundamental group} $\pi_1:=\Lambda_W/\Lambda_R$. This
corresponds to choosing a Lie group associated to the Lie algebra $\g$; we call
the two extreme cases $\Lambda=\Lambda_W$ the \emph{simply-connected form} and
$\Lambda=\Lambda_R$ the usual \emph{adjoint form} (e.g. $SL_2$ vs.
$PSL_2$), see e.g. \cite{CP95} Sec. 9.1 or \cite{Lusz94}.

\begin{definition}\label{def_RationalForm}
  For each abelian group $\Lambda$ with $\Lambda_R\subset \Lambda\subset
  \Lambda_W$ we define the \emph{rational form} $U^{\Q(q)}_q(\g,\Lambda)$ over
  the ring of rational functions $\k=\Q(q)$ as follows:\\

  As algebra, let $U^{\Q(q)}_q(\g,\Lambda)$ be generated by the group ring
  $\k[\Lambda]$ spanned by $K_\lambda,\lambda\in\Lambda$ and additional
  generators $E_{\alpha_i},F_{\alpha_i}$ for each simple root $\alpha_i, i\in
  I$ with relations:
  \begin{align*}
    K_{\lambda}E_{\alpha_i}K_{\lambda}^{-1}
    &=q^{(\lambda,\alpha_i)}E_{\alpha_i}
    ,\; \forall\lambda\in\Lambda
    \qquad \mbox{\emph{(group action)}}\\
    K_{\lambda}F_{\alpha_i}K_{\lambda}^{-1}
    &=\bar{q}^{(\lambda,\alpha_i)}F_{\alpha_i}
    ,\;\forall\lambda\in\Lambda
    \qquad \mbox{\emph{(group action)}}\\
    [E_{\alpha_i},F_{\alpha_j}]
    &=\delta_{i,j}\cdot\frac{K_{\alpha_i}-K_{\alpha_i}^{-1}}
      {q_{\alpha_i}-q_{\alpha_i}^{-1}}
    \qquad \mbox{\emph{(linking)}}
    \end{align*}
    and two sets of \emph{Serre-relations} for any $i\neq j\in I$
    \begin{align*}
    \sum_{r=0}^{1-a_{ij}} (-1)^r
    \begin{bmatrix}1-a_{ij}\\ r\end{bmatrix}_{q^{d_i}}
    E_{\alpha_i}^{1-a_{ij}-r}E_{\alpha_j}E_{\alpha_i}^{r}
    &=0\\
    \sum_{r=0}^{1-a_{ij}} (-1)^r 
    \begin{bmatrix}1-a_{ij}\\ r\end{bmatrix}_{\bar{q}^{d_i}}
    F_{\alpha_i}^{1-a_{ij}-r}F_{\alpha_j}F_{\alpha_i}^{r}
    &=0
  \end{align*}
  where $\bar{q}:=q^{-1}$, the $\begin{bmatrix}n\\k\end{bmatrix}_{q^{d_i}}$
  are the quantum binomial coefficients (see \cite{Lusz94} Sec. 1.3)
   and by definition
  $q^{(\alpha_i,\alpha_j)}=(q^{d_i})^{a_{ij}}$.
  As a coalgebra, let the \emph{coproduct} $\Delta$, the \emph{counit}
  $\epsilon$ and the \emph{antipode} $S$ be defined on the group-Hopf-algebra
  $\k[\Lambda]$ as usual 
  $$\Delta(K_\lambda)=K_\lambda\otimes K_\lambda\qquad
  \epsilon(K_\lambda)=1\qquad
  S(K_\lambda)=K_{\lambda}^{-1}=K_{-\lambda}$$ 
  and on the
  additional generators $E_{\alpha_i},F_{\alpha_i}$ for each simple root
  $\alpha_i, i\in I$ as follows:
  \begin{align*}
    \Delta(E_{\alpha_i})
    =E_{\alpha_i}\otimes K_{\alpha_i}+1\otimes E_{\alpha_i} 
    &\qquad \Delta(F_{\alpha_i})
    =F_{\alpha_i}\otimes 1+K_{\alpha_i}^{-1}\otimes F_{\alpha_i}\\
    S(E_{\alpha_i})=-E_{\alpha_i}K_{\alpha_i}^{-1}
    &\qquad S(F_{\alpha_i})=-K_{\alpha_i}F_{\alpha_i}\\
    \epsilon(E_{\alpha_i})=0
    &\qquad \epsilon(F_{\alpha_i})=0
  \end{align*}
\end{definition}
\begin{theorem}[Rational Form]\label{thm_RationalForm}
  $U^{\Q(q)}_q(\g,\Lambda)$ is a Hopf algebra over the field
  $\k=\Q(q)$. For arbitrary $\Lambda$ using the construction in 
  Theorem \nref{thm_ExtensionOfScalars} we have
  $$U^{\Q(q)}_q(\g,\Lambda)
    =\k[\Lambda]\ltimes_{\k[\Lambda_R]}U^{\Q(q)}_q(\g,\Lambda_R)$$
  Moreover, we have a \emph{triangular decomposition}: Consider the
  subalgebras  $U^{\Q(q),+}_q$ generated by the $E_{\alpha_i}$ and 
  $U^{\Q(q),-}_q$  generated by the $F_{\alpha_i}$ and 
  $U^{\Q(q),0}_q=\k[\Lambda]$ spanned by  the $K_{\lambda}$. Then
  multiplication in $U_q^{\Q(q)}$ induces an isomorphism of vector spaces:
  $$U^{\Q(q),+}_q\otimes U^{\Q(q),0}_q\otimes U^{\Q(q),-}_q
  \stackrel{\cong}{\longrightarrow}U^{\Q(q)}_q$$
\end{theorem}
\begin{proof}
  The case of the adjoint form $\Lambda=\Lambda_R$ is classical, see e.g.
  \cite{Jan03} II, H.2 \&  H.3. In principle, this and later proofs work
  totally analogous for arbitrary  $\Lambda_R\subset \Lambda\subset \Lambda_W$,
  but to connect them directly to results in literature without repeating
  everything, we deduce the case of arbitrary $\Lambda$ from $\Lambda=\Lambda_R$
  and the construction in Section \nref{sec_ExtensionOfScalars}:\\
 
  Let $\k=\Q(q)$, take $\Lambda_R\subset \Lambda\subset \Lambda_W$ an abelian
  group and let $H=U_q^{\Q(q)}(\g):=U_q^{\Q(q)}(\g,\Lambda_R)$ be the adjoint
  form with smallest $\Lambda=\Lambda_R$. Define an action $\rho$ of
  $\k[\Lambda]$ on $H$ given by
  \begin{align*}
    \rho(K_{\lambda}\otimes E_{\alpha_i})
    &=q^{(\lambda,\alpha_i)}E_{\alpha_i}\\
    \rho(K_{\lambda}\otimes F_{\alpha_i})
    &=\bar{q}^{(\lambda,\alpha_i)}F_{\alpha_i}
  \end{align*}
Then certainly the restriction of this action to $\k[\Lambda_R]\subset
\k[\Lambda]$ is the adjoint action in Definition \nref{def_RationalForm} for
$H=U_q^{\Q(q)}(\g,\Lambda_R)$ and the restriction to $\k[\Lambda_R]\subset H$ is
trivial ($\Lambda$ is an abelian group). Hence we can apply extension of
scalars by an abelian group in Theorem \nref{thm_ExtensionOfScalars} and yield
a Hopf algebra 
$$H_\Lambda:=\k[\Lambda]\ltimes_{\k[\Lambda_R]} H$$
Denote the elements $K_\lambda\otimes_{\k[\Lambda_R]} 1$ by $K_\lambda$,
especially for $\alpha\in \Lambda_R$ we have $K_\alpha=1\otimes_{\k[\Lambda_R]}
K_\alpha$ with $K_\alpha\in U_q^{\Q(q)}(\g,\Lambda_R)$. Then it is clear that
these elements fulfill the relations given in the previous Definition
of $U_q^{\Q(q)}(\g,\Lambda)$ for general $\Lambda$. It follows from the
triangular decomposition of $H$ that this is an isomorphism of vector spaces as
$$\k[\Lambda]\otimes_{\k[\Lambda_R]} \k[\Lambda_R]\cong \k[\Lambda]$$
Especially, $U_q^{\Q(q)}(\g,\Lambda)$ defined above is a Hopf algebra with a
triangular decomposition as vector spaces
$$U^{\Q(q),+}_q\otimes U^{\Q(q),0}_q\otimes U^{\Q(q),-}_q
  \stackrel{\cong}{\longrightarrow}U^{\Q(q)}_q$$
with $U^{\Q(q),0}_q\cong \k[\Lambda]$ and $U^{\Q(q),\pm}_q$ independent of
the choice of $\Lambda$.
\end{proof}

A tool of utmost importance has been introduced by Lusztig, see \cite{Jan03}
H.4:

\begin{definition}
  Fix a reduced expression $s_{i_1}\cdots s_{i_t}$ of the longest element in the
  Weyl group $\W(\g)$ in terms of reflections $s_i$ on simple roots
  $\alpha_i$. 
  \begin{enumerate}
   \item There
  exist algebra automorphisms $T_i:U_q^{\Q(q)}(\g,\Lambda)\to
  U_q^{\Q(q)}(\g,\Lambda)$, such that the
  action restricted to $K_\lambda\in\Lambda\subset \Lambda_W$ is the reflection
  of the weight $\lambda$ on $\alpha_i$.
    \item Every positive root $\beta$ has a unique expression
    $\beta=s_{i_1}\cdots s_{i_{k-1}}\alpha_{i_k}$ for some index $k$. This
    defines a
    total ordering on the set of positive roots $\Phi^+$ and the reversed
    ordering on $\Phi^-$. 
    Define the \emph{root vectors} for a root $\beta\in\Phi^+$ by 
    $$E_{\beta}:= T_{i_1}\cdots T_{i_{k-1}}E_{\alpha_{i_k}}$$
    $$F_{\beta}:= T_{i_1}\cdots T_{i_{k-1}}F_{\alpha_{i_k}}$$
  \end{enumerate}
\end{definition}
With these definitions, Lusztig establishes a PBW vector space basis:
\begin{theorem}[PBW-basis]\label{thm_RationalFormPBW}
  Multiplication in $U_q^{\Q(q)}$ induces an isomorphism of $\k$-vector
  spaces for the field $\k=\Q(q)$
  $$\k[\Lambda] 
  \bigotimes_{\alpha\in \Phi^+} \k[E_\alpha]
  \bigotimes_{-\alpha\in \Phi^-} \k[F_\alpha]
  \stackrel{\cong}{\longrightarrow} U_q^{\Q(q)}(\g,\Lambda)$$
  where the orderings on $\Phi^+,\Phi^-$ are as above.
\end{theorem}
\begin{proof}
  The adjoint case $\Lambda=\Lambda_R$ is classical and in \cite{Jan03}
  H.4. Note that using the relations between $K_\alpha,E_\alpha$ all $K$'s
  can be sorted to the left side. The case of arbitrary
  $\Lambda$ could be derived totally analogously, but it also follows directly
  from the presentation as extension in Theorem
  \nref{thm_RationalForm}. Namely, we have by construction isomorphisms of
  vector spaces
  \begin{align*}
  U_q^{\Q(q)}(\g,\Lambda)
  &\cong \k[\Lambda]\otimes_{\k[\Lambda_R]} U_q^{\Q(q)}(\g,\Lambda_R)\\
  &\cong \k[\Lambda]\otimes_{\k[\Lambda_R]}\k[\Lambda_R] 
  \bigotimes_{\alpha\in\Phi^+} \k[E_\alpha] 
  \bigotimes_{-\alpha\in \Phi^-} \k[F_\alpha]\\
  &\cong \k[\Lambda] 
  \bigotimes_{\alpha\in \Phi^+} \k[E_\alpha] 
  \bigotimes_{-\alpha\in \Phi^-}\k[F_\alpha]
  \end{align*}
\end{proof}

\subsection{Two integral forms}
Next we define two distinct \emph{integral forms}
$U_q^{\Z[q,q^{-1}],\K}(\g,\Lambda)$ and $U_q^{\Z[q,q^{-1}],\L}(\g,\Lambda)$.
These are $\Z[q,q^{-1}]$-subalgebras of $U_q^{\Q(q)}(\g,\Lambda)$ which are
after extension of scalars $\otimes_{\Z[q,q^{-1}]}\Q(q)$ isomorphic to
$U_q^{\Q(q)}(\g,\Lambda)$ as $\Q(q)$-algebras.

\begin{definition} 
(cf. \cite{CP95} Sec. 9.2 and 9.3)\label{def_IntegralForms} Recall
$q_\alpha:=q^{d_\alpha}=q^{(\alpha,\alpha)/2}$.
  \begin{itemize}
    \item The so-called \emph{unrestricted integral form}
    $U_q^{\Z[q,q^{-1}],\K}(\g,\Lambda)$ is generated as a $\Z[q,q^{-1}]$-algebra
    by $\Lambda$ and the following elements in
    $U_q^{\Q(q)}(\g,\Lambda)^{+,-,0}$ 
    $$E_{\alpha},\;F_{\alpha},\;
    \frac{K_{\alpha_i}-K_{\alpha_i}^{-1}}
    {q_{\alpha_i}-q_{\alpha_i}^{-1}}\qquad \forall\alpha\in\Phi^+,i\in I$$ 
    We use the superscript $\K$ in honor of Victor Kac, who has defined and
    studied it in characteristic $p$ in 1967 with Weisfeiler and the present
    form in 1990--1992 with DeConcini and Procesi.
  \item The so-called \emph{restricted integral form} 
    $U_q^{\Z[q,q^{-1}],\L}(\g,\Lambda)$ is generated as a $\Z[q,q^{-1}]$-algebra
    by $\Lambda$ and the following elements in $U_q^{\Q(q)}(\g,\Lambda)^\pm$
    called
    \emph{divided powers}  
    $$E_{\alpha}^{(r)}:=\frac{E_{\alpha}^r}
      {\prod_{s=1}^r\frac{q_\alpha^s-q_\alpha^{-s}}
      {q_\alpha-q_{\alpha}^{-1}}},
    \;F_{\alpha}^{(r)}:=\frac{F_{\alpha}^r}
      {\prod_{s=1}^r\frac{\bar{q}_\alpha^{s}-\bar{q}_\alpha^{-s} }
      {\bar{q}_\alpha-\bar{q}_{\alpha}^{-1}}}
    \qquad \forall \alpha\in\Phi^+,r\geq 0$$
    and by the following elements in $U_q^{\Q(q)}(\g,\Lambda)^0$:
    $$K_{\alpha_i}^{(r)}=\Ksymb{K_{\alpha_i}}{0}{r}
      :=\prod_{s=1}^r\frac{K_{\alpha_i}q_{\alpha_i}^{1-s}
      -K_{\alpha_i}^{-1}q_{\alpha_i }^{s-1}}
      {q_{\alpha_i}^s-q_{\alpha_i}^{-s}}\qquad i\in I$$ 
    We use the superscript $\L$ in honor of Georg Lusztig, who has
    defined and studied it in 1988--1990.
  \end{itemize}
\end{definition}
\begin{theorem}\label{thm_IntegralForm}
 The Lusztig quantum group $U_q^{\Z[q,q^{-1}],\L}(\g,\Lambda)$ is a Hopf
 algebra over the ring
 $\k=\Z[q,q^{-1}]$ and is an integral forms for $U_q^{\Q(q)}(\g,\Lambda)$.
  Hereby for arbitrary $\Lambda$ we have again by  
  Theorem \nref{thm_ExtensionOfScalars}
 $$U_q^{\Z[q,q^{-1}],\L}(\g,\Lambda)
 \cong \k[\Lambda]\ltimes_{\k[\Lambda_R]}U^{\Q(q)}_q(\g,\Lambda_R)$$
\end{theorem}
A similar result holds for the Kac integral form, see Chari \cite{CP95} Sec.
9.2. Generators and relations for simply-laced $\g$ are discussed in
\cite{CP95} Thm. 9.3.4. The
proof, that the $U_q^{\Z[q,q^{-1}],\K},U_q^{\Z[q,q^{-1}],\L}$ are integral
forms for $U_q^{\Q(q)}$ follows immediately from the remarkable knowledge of a
PBW-basis:

\begin{theorem}[PBW-Basis]\label{thm_IntegralFormPBW}
  For the Lusztig integral form $U_q^{\Z[q,q^{-1}],\L}$ over the commutative
  integral domain $\k=\Z[q,q^{-1}]$,  multiplication induces an isomorphism
  of $\k$-modules

  $$\k[\Lambda/2\Lambda_R]
  \bigotimes_{i\in I}\left(\bigoplus_{r\geq 0} K_{\alpha_i}^{(r)}\k\right)
  \bigotimes_{\alpha\in \Phi^+}\left(\bigoplus_{r\geq 0}E_{\alpha}^{(r)}\k\right)
  \bigotimes_{-\alpha\in \Phi^-}\left(\bigoplus_{r\geq 0} F_{\alpha}^{(r)}\k\right)
  \stackrel{\cong}{\longrightarrow} U_q^{\Z[q,q^{-1}],\L}(\g,\Lambda)$$

  Especially, the Lusztig integral form is free as $\k$-module. Note that the
  group algebra $\k[\Lambda/2\Lambda_R]$ is not contained in
  $U_q^{\Z[q,q^{-1}],\L}$ as an algebra, just  as $\k$-module!
\end{theorem}  
\begin{proof}
  For $\Lambda=\Lambda_R$ it is proven by Lusztig (see \cite{Jan03} H.5) that 
  the sorted monomials in the root vectors $E_\alpha$ resp. $F_\alpha$ for
  $\alpha\in \Phi^+$ form a basis of $U_q^{\Z[q,q^{-1}],\L,\pm}$ as a module
  over the commutative ring  $\k=\Z[q,q^{-1}]$ and that the products
  $\prod_{i\in I} K_{\alpha_i}^{\delta_i} K_{\alpha_i}^{(r_i)}$ with
  $\delta_i\in\{0,1\},r_i\geq 0$ form a $\k$-basis of 
  $U_q^{\Z[q,q^{-1}],\L,0}$. The latter statement is by the commutativity
  equivalent to the statement $U_q^{\Z[q,q^{-1}],\L,0}\cong
  \k[\Lambda/2\Lambda_R] \bigotimes_{i\in I}(\bigoplus_{r\geq 0} K_{\alpha_i}^{(r)}\k)$. 
  Note that the PBW-basis theorem does not follows from the PBW-basis  of the
  rational form. Rather, the proof proceeds parallel and roughly uses  that the
  $T_i$ preserve the chosen generator set
  of $U_q^{\Z[q,q^{-1}],\L}$.\\
  
  The case of arbitrary
  $\Lambda$ could be derived totally analogously, but it also follows directly
  from the presentation in Theorem \nref{thm_IntegralForm}, and is proven
  as in the proof of Theorem \nref{thm_RationalFormPBW}:
  \begin{align*}
  &U_q^{\Z[q,q^{-1}],\L}(\g,\Lambda)\\
  &\cong \k[\Lambda]\otimes_{\k[\Lambda_R]}
    U_q^{\Z[q,q^{-1}],\L}(\g,\Lambda_R)\\
  &\cong \k[\Lambda]\otimes_{\k[\Lambda_R]}\k[\Lambda_R/2\Lambda_R]
    \bigotimes_{i\in I}\left(\bigoplus_{r\geq 0} K_{\alpha_i}^{(r)}\k\right)
    \bigotimes_{\alpha\in  \Phi^+}\left(\bigoplus_{r\geq 0} E_{\alpha}^{(r)}\k\right)
    \bigotimes_{-\alpha\in \Phi^-}\left(\bigoplus_{r\geq 0} F_{\alpha}^{(r)}\k\right)\\
  &\cong \k[\Lambda/2\Lambda_R]
  \bigotimes_{i\in I}\left(\bigoplus_{r\geq 0} K_{\alpha_i}^{(r)}\k\right)
  \bigotimes_{\alpha\in \Phi^+}\left(\bigoplus_{r\geq 0} E_{\alpha}^{(r)}\k\right)
  \bigotimes_{-\alpha\in \Phi^-}\left(\bigoplus_{r\geq 0} F_{\alpha}^{(r)}\k\right)
  \end{align*}
\end{proof}

\subsection{Specialization to roots of unity}

Next we define the \emph{restricted specialization} $U_q^\L(\g,\Lambda)$.
It is a complex Hopf algebra depending on a specific choice $q\in \C^\times$.

\begin{definition} 
(cf. \cite{CP95} Sec. 9.2 and 9.3)\label{def_Specialization}
    The infinite-dimensional complex Hopf algebra 
    $U_q^\L(\g,\Lambda)$ is defined by 
    $$U_q^\L(\g,\Lambda):=U_q^{\Z[q,q^{-1}],\L}\otimes_{\Z[q,q^{-1}]} \C_q$$
    where (by slight abuse of notation) $\C_q=\C$ with the 
    $\Z[q,q^{-1}]$-module structure defined by the specific value
    $q\in\C^\times$.
\end{definition}

Note that we have a PBW-basis in Theorem \nref{thm_IntegralFormPBW}, which
especially shows $U_q^{\Z[q,q^{-1}],\L}$ is free as a $\Z[q,q^{-1}]$-module.
Hence the specialization has an induced vector space basis, the impact of the
specialization is to severely modify the algebra structure, such that e.g.
former powers may become new algebra generators. 

\begin{corollary}\label{cor_SpecializationPBW}
  For the Lusztig quantum group $U_q^{\L}$ over $\C$, multiplication induces
  an isomorphism of $\C$-vector spaces:
  $$\C[\Lambda/2\Lambda_R]
  \bigotimes_{i\in I}\left(\bigoplus_{r\geq 0} K_{\alpha_i}^{(r)}\C\right)
  \bigotimes_{\alpha\in \Phi^+}\left(\bigoplus_{r\geq 0}E_{\alpha}^{(r)}\C\right)
  \bigotimes_{-\alpha\in \Phi^-}\left(\bigoplus_{r\geq 0} F_{\alpha}^{(r)}\C\right)
  \stackrel{\cong}{\longrightarrow} U_q^{\L}(\g,\Lambda)$$
  This PBW-basis will we refined in Lemma \nref{lm_SpecializationPBW}.
\end{corollary}

\begin{example}
    For $q=1$ we have a cosmash-product
    $$U_1^\L(\g,\Lambda)\cong \C[\Lambda/2\Lambda]\ltimes U(\g)$$
\end{example}

\section{First properties of the specialization}\label{sec_quantumgroup}

For the rest of the article we assume $q\in\C^\times$ a primitive $\ell$-th root
of unity without restrictions on $\ell$. We study the
infinite-dimensional Lusztig quantum group $U_q^{\L}(\g,\Lambda)$ from
Definition \nref{def_Specialization}, which is a Hopf algebra over $\C$. It was
defined as a specialization of the Lusztig integral form in Definition
\nref{def_IntegralForms} and hence shares the explicit vector space basis given
by Theorem \nref{thm_IntegralFormPBW}.

\subsection{The zero-part}

The zero-part $u_q^\L(\g,\Lambda)^0$ in the triangular decomposition uses
different arguments than the quantum Borel parts. Recall from Corollary
\nref{cor_SpecializationPBW}, that multiplication in $U_q^\L$ induces an
isomorphism of vector spaces

$$\bigotimes_{i\in I} \C[K_{\alpha_i}]/(K_{\alpha_i}^{2\ell_i})
\bigotimes_{i\in I}\left(\bigoplus_{r\geq 0} K_{\alpha_i}^{(r)}\C\right)
=\C[\Lambda/2\Lambda_R]\bigotimes_{i\in I}\left(\bigoplus_{r\geq 0} K_{\alpha_i}^{(r)}\k\right)
\stackrel{\cong}{\longrightarrow} U_q^\L(\g,\Lambda)^0$$

We want to determine the algebra structure of $U_q^{\L,0}$. It is clear from
the definition, that $U_q^{\L,0}$ is a commutative, cocommutative complex Hopf
algebra. Note that by the theorem of Kostant-Cartier (see e.g. \cite{Mont93}
Sec. 5.6) this already
implies it is of the form $\C[G]\otimes U(\mathfrak{h})$
with group of grouplikes $G=G(U_q^{\L,0})$ and $\mathfrak{h}$ an abelian Lie
algebra.

\begin{theorem}\label{thm_ZeroPart}
  With $\ell_i=ord(q_\alpha^2)$ as always we have an isomorphism of Hopf
  algebras
 $$U_q^{\L}(\g,\Lambda)^0\cong \C[\Lambda/2\Lambda^{(\ell)}]\otimes
  U(\mathfrak{h})$$
  with generators 
  $$K_\lambda,\lambda\in \Lambda,\qquad
  H_{\alpha_i}:=\frac{K_{\alpha_i}^{2\ell_i}-1}{q_\alpha^{2\ell_\alpha}-1}$$
\end{theorem}
\begin{proof}		
  Modulo elements $K_{\alpha_i}$, the expression for $H_{\alpha_i}$ can be rewritten as $K_{\alpha_i}^{(\ell_i)}$, which shows that that we have a bijection. The coproduct of the $H_{\alpha_i}$ is easily calculated from definition to be
  $$\Delta(H_{\alpha_i})=K_{\alpha_i}^{2\ell_i}\otimes H_{\alpha_i}+H_{\alpha_i}\otimes 1 =1\otimes H_{\alpha_i}+H_{\alpha_i}\otimes 1$$		
  
  In \cite{Len17} Theorem 3.1 we have computed explicit expressions for $H_{\alpha_i}$ in terms of $K_{\alpha_i}, K_{\alpha_i}^{(\ell_i)}$. Its action on simple modules exhibits a characteristic weight shift. 
\end{proof}

\subsection{The coradical}

The following assertion is known under various restrictions and follows from a
standard argument, see e.g. \cite{Mont93} Lemma 5.5.5. We include it for
completeness in the case of arbitrary $\ell$. It would be interesting to
determine the full coradical filtration.

\begin{lemma}\label{lm_SpecializationCoradical}
  The coradical of the infinite-dimensional Hopf algebra $U_q^\L(\g,\Lambda)$
  is $\C[\Lambda]$. Especially the Hopf algebra is pointed with group
  of grouplikes $\Lambda$. 
\end{lemma}
\begin{proof}
  Consider the (very coarse) coalgebra $\N$-grading induced by setting
  $deg(E_\alpha^{(r)})=deg(F_\alpha^{(r)})=r$ and $\deg(x)=0$ for $x\in
  U_q^{\L,0}(\g,\Lambda)$. By \cite{Sw69} Prop. 11.1.1 this already implies
  that the coradical is contained in $U_q^{\L,0}(\g,\Lambda)$. We have shown in 
  Theorem \nref{thm_ZeroPart} that 
  $$U_q^{\L,0}(\g,\Lambda)\cong \C[\Lambda]\otimes U(\mathfrak{h})$$
  Hence the coradical is indeed the group algebra $\C[\Lambda]$. This
  especially shows that there are no other grouplikes than $\Lambda$.
\end{proof}

\subsection{The positive part}

A curious aspect of this article is, that via Lusztig's PBW-basis of root
vectors, we have complete control over the vector space $U_q^{\L,+}$, also in
degenerate cases. The involved question we addressed is the algebra and Hopf
algebra structure. We start in this section by some preliminary observations
in this direction.\\

From the PBW-basis in $U_q^{\Z[q,q^{-1}],\L}$ we have
already determined in Corollary \nref{cor_SpecializationPBW} that multiplication
induces an isomorphism of vector spaces
  $$\C[\Lambda/2\Lambda_R]
  \bigotimes_{i\in I}\left(\bigoplus_{r\geq 0} K_{\alpha_i}^{(r)}\C\right)
  \bigotimes_{\alpha\in \Phi^+}\left(\bigoplus_{r\geq 0}E_{\alpha}^{(r)}\C\right)
  \bigotimes_{-\alpha\in \Phi^-}\left(\bigoplus_{r\geq 0} F_{\alpha}^{(r)}\C\right)
  \stackrel{\cong}{\longrightarrow} U_q^{\L}(\g,\Lambda)$$
The aim of the next theorem is to use the knowledge of the zero-part in the
previous section and a straight-forward-calculation to incorporate at least
part of the algebra relations that hold specifically in the specialization,
without any restrictions on $\ell$:

\begin{lemma}\label{lm_SpecializationPBW}
  Let $q$ be a primitive $\ell$-th root of unity. Then multiplication
  in
  $U_q^\L(\g,\Lambda)$ induces an isomorphism of vector spaces, which
  restricts on each tensor factor to an injection of algebras: 
  $$\C[\Lambda/2\Lambda_R^{(\ell)}]\otimes U(\mathfrak{h})  
  \bigotimes_{\alpha\in\Phi^+}
  \left(\C[E_{\alpha}]/(E_\alpha^{\ell_\alpha})
  \otimes \C[E_{\alpha}^{(\ell_\alpha)}]\right)
  \bigotimes_{-\alpha\in\Phi^-}
  \left(\C[F_{\alpha}]/(F_\alpha^{\ell_\alpha})
  \otimes \C[F_{\alpha}^{(\ell_\alpha)}]\right)
  \stackrel{\cong}{\longrightarrow} U_q^{\L}(\g,\Lambda)$$
\end{lemma}
\begin{proof}
  By Corollary \nref{cor_SpecializationPBW}, multiplication in
  $U_q^\L$ induces an isomorphism of vector spaces
  $$\C[\Lambda/2\Lambda_R]
  \bigotimes_{i\in I}\left(\bigoplus_{r\geq 0} K_{\alpha_i}^{(r)}\C\right)
  \bigotimes_{\alpha\in \Phi^+}\left(\bigoplus_{r\geq 0}E_{\alpha}^{(r)}\C\right)
  \bigotimes_{-\alpha\in \Phi^-}\left(\bigoplus_{r\geq 0} F_{\alpha}^{(r)}\C\right)
  \stackrel{\cong}{\longrightarrow} U_q^{\L}(\g,\Lambda)$$
  We clarified in Theorem \nref{thm_ZeroPart} the zero-part $U_q^{\L,0}$, so we
  get an isomorphism 
  $$\C[\Lambda/2\Lambda^{(\ell)}]\otimes
  U(\mathfrak{h})\bigotimes_{\alpha\in \Phi^+}\left(\bigoplus_{r\geq 0}E_{\alpha}^{(r)}\C\right)
  \bigotimes_{-\alpha\in \Phi^-}\left(\bigoplus_{r\geq 0}F_{\alpha}^{(r)}\C\right)
  \stackrel{\cong}{\longrightarrow} U_q^{\L}(\g,\Lambda)$$
 
  We next turn our attention to the algebra generated for a fixed root
  $\alpha\in\Phi^+$ by all elements
  $E_\alpha^{(r)}=E^r/[r]_{q_\alpha}$ in the
  specialization to a primitive $\ell$-th root of unity (respectively for $F$).
  Since
  $[r]_{q_\alpha}=0$ iff $\ell_\alpha:=ord(q_\alpha^2)|r$, it is clearly
  isomorphic to
  $$\bigoplus_{r\geq 0} E_\alpha^{(r)}\C
  \cong \begin{cases}
    \C[E_\alpha]/(E_\alpha^{\ell_i})\otimes\C[E_\alpha^{(\ell_\alpha)}]
      &\ell_\alpha>1\\
    \C[E_\alpha] & \ell_\alpha=1
  \end{cases}$$
  This yields the asserted isomorphism.
\end{proof}

For later use we observe:
\begin{lemma}\label{lm_commute}
  Assume for some $\alpha,\beta\in\Phi^+$ holds $E_\alpha
  E_\beta=q^{(\alpha,\beta)}E_\beta E_\alpha$ already in $U_q^{\Q(q)}$, then
  $E_\alpha^{(k)} E_\beta^{(k')}=q^{(\alpha,\beta)kk'}E_\beta^{(k')}
  E_\alpha^{(k)}$. Assume $\g$ of rank $2$ and $\alpha=\alpha_i$ simple and
  $\alpha+\beta\not\in\Phi^+$ then the assumption holds except for the
  following cases:
  \begin{align*}
    B_2:\;& (\alpha_{112},\alpha_2)\\
    G_2:\;& (\alpha_{11122},\alpha_2),(\alpha_{112},\alpha_2),
    (\alpha_1,\alpha_{11122} )
  \end{align*}
  Note that $(\alpha_1,\alpha_{11122})$ is the reflection of
  $(\alpha_{112}, \alpha_2)$ on $\alpha_{12}$.
\end{lemma}
\begin{proof}
 The first assertion is trivially checked in $U_q^{\Q(q)}$, where divided
  powers can be written as powers. The second assertion follows by inspecting
  \cite{Lusz90b} Sec. 5.2. Note that the other exceptional cases with $\alpha$
  not simple would follow easily by Weyl reflection. Precisely these exceptions
  will generate the dual root system in Lemma \nref{lm_gell}.
\end{proof}

\section{The small quantum group for arbitrary
\texorpdfstring{$q$}{q}}\label{sec_smalluq}

Lusztig has in \cite{Lusz90b} Thm 8.10. discovered a remarkable homomorphism 
from $U_q^\L(\g)$ to the ordinary universal enveloping Hopf algebra $U(\g)$
$$U_q^\L(\g,\Lambda)\xrightarrow{Frob}U(\g)$$
whenever $\ell$ is odd and for $\g=G_2$ not divisible by $3$. He described
the kernel in terms of an even more remarkable finite-dimensional
Hopf algebra $u_q(\g)$. This Hopf algebra has under the name
\emph{Frobenius-Lusztig kernel} triggered the development of the
theory and several far-reaching classification results on
finite-dimensional pointed Hopf algebras and Nichols algebras in the past
$\sim$20 years.\\

Lusztig's definition and PBW-basis for $u_q^\L(\g)$ is without restrictions on
the order of $q$. It does however not describe the structure of $u_q^\L(\g)$ as
an algebra. After reviewing Lusztig's definition and theorem, we will
describe the algebra in terms of Nichols algebras and clarify its structure.
Especially $u_q^\L(\g)$ does for small order of $q$ {\bf not} coincide with the
usual description in terms of generators and relations, which we will denote
$u_q(\g)$ -- in the \emph{exotic case} $\g=G_2,q=\pm i$ it will be even of
larger rank, namely $A_3$. 

\begin{definition}[\cite{Lusz90b} Sec. 8.2]\label{def_uqL}
    Let $u_q^\L(\g,\Lambda)\subset U_q^\L(\g,\Lambda)$ be the subalgebra
    generated by $\Lambda$ and all $E_\alpha,F_\alpha$ with $\alpha\in\Phi^+$
    such that $\ell_\alpha>1$. 
\end{definition}
Note this implicit definition does not give the algebra structure, but the
vector space is well understood using Lusztig reflection operators:
\begin{theorem}[\cite{Lusz90b} Thm. 8.3]\label{thm_Lusztiguq}
    $u_q^\L(\g,\Lambda)$ is a Hopf subalgebra and multiplication
    in $u_q^\L$ defines an isomorphism of vector spaces:
      $$\C[\Lambda/2\Lambda_R^{(\ell)}]  
      \bigotimes_{\alpha\in\Phi^+,\ell_\alpha>1}
	\C[E_{\alpha}]/(E_\alpha^{\ell_\alpha})
      \bigotimes_{-\alpha\in\Phi^-,\ell_\alpha>1}
	\C[F_{\alpha}]/(F_\alpha^{(\ell_\alpha)})
      \stackrel{\cong}{\longrightarrow} u_q^{\L}$$
    Especially $u_q^\L$ is of finite dimension $|\Lambda|\cdot
    \prod_{\alpha\in \Phi^+}\ell_\alpha^2$
\end{theorem}
\begin{proof}
    In \cite{Lusz90b} Thm 8.3 iii) Lusztig proved for $\Lambda=\Lambda_R$
    and without restrictions on $\ell$ that $u_q^\L$ has a PBW-basis consisting
    of $\Lambda_R$ and all $E_\alpha^{(r)},F_\alpha^{(r)}$ with $r<\ell_\alpha$.
    Note that this set is empty for $\ell_\alpha=1$. We've proven as part of
    Lemma \nref{lm_SpecializationPBW} the simple fact that $\C[E_\alpha]$
    consists precisely of all $E_\alpha^{(r)}$ with $r<\ell_\alpha$. This shows
    the claim for $\Lambda=\Lambda_R$. The case of arbitrary $\Lambda$
    via $\C[\Lambda]\otimes_{\C[\Lambda_R]}$ follows again from the
    presentation in Theorem \nref{thm_IntegralForm}.
\end{proof}

\begin{definition}
  Assume that $\ord(q^2)>d_\alpha$ for all $\alpha\in\Phi$. Let $V$ be the vector space with basis $E_{\alpha_1},\ldots,E_{\alpha_n}$ with diagonal braiding $\chi(\alpha_i,\alpha_j)=q^{(\alpha_i,\alpha_j)}$. Then we denote the Nichols algebra  $\B(V)$ by $u_q(\g)^+$ and the respective double $U(\chi)$ in \cite{Heck10} by $u_q(\g)$.
\end{definition}  
These are the small quantum groups associated to $\g$ as they are usually considered, with PBW generators given in terms of roots of $\g$. We remark that for large order $\ord(q^2)$ the usual  quantum Serre relations and truncation relations are defining relations. A set of defining relations for $u_q(\g)^+$ for general $q$ is given in \cite{An13}.\\

The following theorem compares $u_q^\L,u_q$. For sufficiently large order of $q$ the two definitions coincide (which seems to be well-known, but we prove
it nevertheless), but in other cases the result gives an explicit description in terms of a different $u_q$ and thereby determine its root system:

\begin{theorem}\label{thm_smalluq}
    For $\ord(q^2)>d_\alpha$ for all $\alpha\in \Phi^+$ we have
  $u_q^\L(\g,\Lambda)\cong  u_q(\g,\Lambda)$. If some 
  $\ord(q^2)\leq d_\alpha$  we can
  express $u_q^\L(\g,\Lambda)^+$ in terms of some ordinary
  $u_q(\g^{(0)},\Lambda)^+$ as
  follows:
  \begin{center}
  \begin{tabular}{ll|llll}
    $q\qquad$ & $\g\qquad$ & $u_q^\L(\g,\Lambda)^+\qquad$ & $\dim$ & primitive
    generators & Comment\\
    \hline
    $\pm 1$ & all & $\C$ & $1$ & none & trivial \\
    $\pm i$ & $B_n$ & $u_q(A_1^{\times n})^+$ & $2^n$ & $E_{\alpha_n},
      E_{\alpha_n+\alpha_{n-1}}, E_{\alpha_n+\alpha_{n-1}+\alpha_{n-2}}, \ldots$
      & short roots \\
    $\pm i$ & $C_n$ & $u_q(D_n)^+$ & $2^{n(n-1)}$ & $E_{\alpha_1},\ldots
      E_{\alpha_{n-1}},E_{\alpha_{n}+\alpha_{n-1}}$ & short roots \\
    $\pm i$ & $F_4$ & $u_q(D_4)^+$ & $2^{12}$ & $E_{\alpha_4},E_{\alpha_3},
    E_{\alpha_3+\alpha_2},E_{\alpha_3+\alpha_2+\alpha_1}$
      & short roots \\
    $\sqrt[3]{1},\sqrt[6]{1}$ & $G_2$ & $u_q(A_2)^+$ & $3^3$ & 
      $E_{\alpha_1},E_{\alpha_1+\alpha_2}$
      & short roots \\
    $\pm i$ & $G_2$ & $u_{\bar{q}}(A_3)^+$ & $2^6$ &
      $E_{\alpha_2},E_{\alpha_1},E_{2\alpha_1+\alpha_2}$  &
      exotic \\
  \end{tabular}
  \end{center}
\end{theorem}
The {\bf exotic case} is special in several instances: It is the only
non-trivial case with 
$\ord(q^2)\lneq d_\alpha$. Since all $\ell_\alpha=2$  all root vectors are
contained in $u_q^{\L,+}$ and the rank even increases because an additional
``premature'' relation $\ad_{E_{\alpha_1}}^2(E_{\alpha_2})=0$ requires a new
algebra generator $E_{2\alpha_1+\alpha_2}$, yielding an $A_3$-root system (which
has also $6$ positive roots). Moreover, the braiding matrix of $u_q^{\L,+}$ is
not the braiding matrix of $u_q(A_3)^+$, rather of the complex conjugate root
of unity $u_{\bar{q}}(A_3)^+$, which is the other choice of a primitive fourth
root of unity. This case will exhibit strange phenomena throughout this article.

The author has compiled in the recent preprint \cite{Len14b} a similar list for affine Lie algebras, where more exotic cases appear, see also Problem \nref{prob_affine}.\\

\begin{proof}[Proof of Theorem \nref{thm_smalluq}] The algebra
$u_q^\L(\g,\Lambda)^+$ is a Hopf algebra in the category of
$\Lambda$-Yetter-Drinfel'd modules. 
The strategy of this proof is the following
\begin{enumerate}[a)]
 \item Prove that for a certain subset $X\subset \Phi^+$ the root vectors
  $E_\alpha,\;\alpha\in X$ are in $u_q^\L(\g,\Lambda)^+$ and consist of
  primitive elements. These elements $X$ are ``guessed'' for each root system at
  this point of the proof and appear in the fifth column of the table.
 \item Determine the braiding matrix of the vector space $V$ spanned by
  $E_\alpha,\;\alpha\in X$ and determine the Nichols algebra $\B(V)$ using
  \cite{Heck09} and especially $\dim(\B(V))$.
 \item Now by the universal property of the Nichols algebra we have a
  surjection from the Hopf subalgebra $H\subset u_q^{\L,+}$
  generated by the $E_\alpha,\;\alpha\in X$ to the Nichols algebra $\B(V)$.
  Note that $u_q^{\L,+}(\g)$ will be a graded algebra for trivial reasons
  except for the exotic case, where we show this explicitly.
  Using Lusztig's PBW-basis of $u_q^{\L,+}$ in the previous theorem we
  show in each case $\dim(u_q^{\L,+})=\dim(\B(V))$ and hence
  $\B(V)=H=u_q^{\L,+}$.
\end{enumerate}
We now proceed to the proof according to the steps outlined above. The trivial
case is if all $\ell_\alpha=1$ i.e. $q^2=1$, the generic case is if all
$\ord(q^{2})>d_\alpha$. For simply-laced $\g$ all
$d_\alpha=1$, so these two cases exhaust all
possibilities. For the non-simply-laced $\g=B_n,C_n,F_4$ we have $d_\alpha=2$
for long roots, hence the condition is also violated for $\ord(q^2)=2$; for the
non-simply-laced $\g=G_2$ we have $d_\alpha=3$ for long roots, hence the
condition is also violated for $\ord(q^2)=2,3$. Thus, we have to check precisely
the cases in the table of the theorem. 
\begin{enumerate}[a)]
 \item We show that the $E_\alpha,\;\alpha\in X$ given in the last column of the
  theorem are primitive. We first show the well-known general fact that if
  $x,y$ are primitive elements with braidings $(x\otimes y)\mapsto
  q_{12}(y\otimes  x)\mapsto q_{12}q_{21}(x\otimes y)$, then $q_{12}q_{21}=1$
  implies the braided commutators $[x,y]:=xy-q_{12}yx$ resp.
  $[y,x]:=yx-q_{21}xy$ are  primitive as well (possibly $=0$):
  \begin{align*}
    \Delta([x,y])
    &=\Delta(xy)-q_{12}\Delta(yx)\\
    &=\left(1\otimes xy+q_{12}y\otimes x+x\otimes y+xy\otimes 1 \right)\\
    &-q_{12}\left(1\otimes yx+q_{21}x\otimes y+y\otimes x+yx\otimes 1\right)\\
    &=1\otimes [x,y]+(1-q_{12}q_{21})(x\otimes y)+[x,y]\otimes 1
  \end{align*}
  Now we check in each case that the $E_\alpha,\;\alpha\in X$ given in the last
  column of the theorem fulfill $\ell_\alpha>1$, hence $E_\alpha\in u_q^{\L,+}$,
  and can be in $U_q^{\Q(q)}$ obtained (inductively) as braided commutators with
  $q_{12}q_{21}=1$, hence they are primitive (the exotic case works different).
  \begin{enumerate}[i)]
    \item In the trivial case $\ell_\alpha=1$ there are no root
    vectors with $\ell_\alpha>1$, hence we chose $X:=\{\}$. In the generic case
    where all
    $\ell_\alpha>d_\alpha$ then all $\ell_\alpha>1$ and all root vectors are in
    $u_q^{\L,+}$, we hence choose $X:=\{\alpha_1,\ldots \alpha_n\}$ (simple
    root vectors are by definition primitive).
    \item For $\g=B_n,q=\pm i$ we have for short roots
    $\ell_\alpha=\ord(q^2)=2$; the
    elements $X:=\{\alpha_n,\alpha_n+\alpha_{n-1},\ldots\}$ are (in fact all)
    short roots. Moreover we can in in $U_q^{\Q(q)}$ inductively obtain   
    $E_{\alpha_n+\alpha_{n-1}+\cdots+\alpha_k}$ by braided commutators of the
    primitive $E_{\alpha_n}\in u_q^{\L,+}$ with the primitive   
    $E_{\alpha_{k\neq n}}\not\in u_q^{\L,+}$. We verify the condition   
    $q_{12}q_{21}=1$ in each step:
    \begin{align*}
     & q^{(\alpha_n+\alpha_{n-1}+\cdots+\alpha_k,\alpha_{k-1})}
     q^{(\alpha_{k-1},\alpha_n+\alpha_{n-1}+\cdots+\alpha_k)}\\
     &=q^{2(\alpha_k,\alpha_{k-1})}=q^{-4}=1 
    \end{align*}
    Note that we have to convince ourselves that the commutator is not
    accidentally zero in the specialization: Reflection easily reduces this   
    case to $B_2$, where we check the commutator explicitly to be $q^2E_{12}\neq
    0$ by \cite{Lusz90b} Sec. 5.2.
    \item For $\g=C_n,q=\pm i$ we have for short roots
    $\ell_\alpha=\ord(q^2)=2$; the
    elements $X:=\{\alpha_1,\alpha_2,\ldots,\alpha_{n-1}+\alpha_n\}$ are    
    short roots. Moreover all $E_{\alpha_{k}}$ are primitive and
    $E_{\alpha_{n-1}+\alpha_n}$ is primitive by applying the case $B_2$ to the
    subsystem generated by $\alpha_{n-1},\alpha_n$.
    \item For $\g=F_4,q=\pm i$ we have for short roots
    $\ell_\alpha=\ord(q^2)=2$; the
    elements
    $X:=\{\alpha_{4},\alpha_3,\alpha_3+\alpha_2,\alpha_{3}+\alpha_2+\alpha_1\}$
    are short roots. Moreover $E_{\alpha_4}$ is primitive and
    $E_{\alpha_3},E_{\alpha_3+\alpha_2},E_{\alpha_3+\alpha_2+\alpha_1}$ are
    primitive by applying the case $B_3$ to the subsystem generated by
    $\alpha_3,\alpha_2,\alpha_1$.
    \item For $\g=G_2,q=\sqrt[3]{1},\sqrt[6]{1}$ we have for short roots
    $\ell_\alpha=\ord(q^2)=3$; the elements
    $X:=\{\alpha_1,\alpha_1+\alpha_2\} $
    are short roots. Moreover $E_{\alpha_1+\alpha_2}$ is primitive since we
    get it in $U_q^{\Q(q)}$ as a braided commutator of
    $E_{\alpha_1},E_{\alpha_2}$   and
    $$q^{(\alpha_1,\alpha_2)}q^{(\alpha_2,\alpha_1)}=q^{-6}=1$$
    We also check the commutator explicitly to be $q^3E_{12}\neq 0$ by
    \cite{Lusz90b} Sec. 5.2.
    \item For $\g=G_2,q=\pm i$ we have for all roots
    $\ell_\alpha=\ord(q^2)=\ord(q^6)=2$. We choose
    $X:=\{\alpha_2,\alpha_1,2\alpha_{1}+\alpha_2\}$. Moreover the elements
    $E_{\alpha_2},E_{\alpha_1}$ are primitive. Checking primitivity of
    $E_{112}$ is more involved. We could
    in principle express $E_{112}$ by definition via reflections in terms of
    $E_1,E_2,F_1,E_1^{(2)},E_2^{(3)}$ but using the relation \cite{Lusz90b}
    Sec. 5.4 (a6) for $k=2$ is more convenient as follows (note that in this
    exotic case there will be no way of expressing $E_{112}$ by
    $E_1,E_2$ as will become clear later in the proof. E.g. Lusztig's relations
    (a3) returns zero for $q=\pm i$):
    $$E_{112}=-q^2(E_2E_1^{(2)}-q^{-6}E_1^{(2)}E_2)-qE_{12}E_1$$
    With this formula we can calculate directly:
    \begin{align*}
     \Delta(E_{12})
      &=\Delta(T_2(E_1))=\Delta(-E_2E_1+q^{-3}E_1E_2)=-\Delta([E_2,E_1])\\
      &=1\otimes E_{12}-(1-q^{-6})E_2\otimes E_1+E_{12}\otimes 1\\
      \Delta(E_{112})
      &=-q^2\left((1\otimes E_2+E_2\otimes 1)
      (1\otimes E_1^{(2)}+qE_1\otimes E_1+E_1^{(2)}\otimes 1)\right.\\
      &\left.-q^{-6}(1\otimes E_1^{(2)}+qE_1\otimes E_1+E_1^{(2)}\otimes 1)
      (1\otimes E_2+E_2\otimes 1)\right)\\
      &-q\left((1\otimes E_{12}-(1-q^{-6})E_2\otimes E_1+E_{12}\otimes 1)
      (1\otimes E_1+E_1\otimes 1)\right)\\
      &=-q^2\left(1\otimes(E_2E_1^{(2)}-q^{-6}E_1^{(2)}E_2)
      +q^{-2}E_1\otimes(E_2E_1-q^{-3}E_1E_2)\right.\\
      &\left.+(1-q^{-12})E_2\otimes E_1^{(2)}   
      +q(E_2E_1-q^{-9}E_1E_2)\otimes E_1
      +(E_2E_1^{(2)}-q^{-6}E_1^{(2)}E_2)\otimes 1\right)\\
      &-q\left(1\otimes E_{12}E_1+q^{-1}E_1\otimes E_{12}
      -(1-q^{-6})E_2\otimes E_1^2\right.\\
      &\left.-q^2(1-q^{-6})E_2E_1\otimes E_1	
      +E_{12}\otimes E_1+E_{12}E_1\otimes E_1\right)\\
      &=1\otimes E_{112}+E_{112}\otimes 1
    \end{align*}
\end{enumerate}
  \item We next determine the braiding matrix and hence from \cite{Heck09} the
  root system and dimension of the Nichols algebra $\B(V)$ of the vector space
  $V$ generated by all $E_{\alpha},\;\alpha\in X$ defined above. This yields
  the information in the third and fourth row of the table. We again proceed
  case-by-case:
  \begin{enumerate}[i)]
    \item For the trivial case we defined $X:=\{\}$ so the braiding matrix on
    $V=0$ is trivial, hence the Nichols algebra is
    $1$-dimensional. For the generic case we defined $X:=\{\alpha_1,\ldots
    \alpha_n\}$ so the braiding matrix
    is $q_{ij}=q^{(\alpha_i,\alpha_j)}=q^{d_ia_{ij}}$. All $\ell_\alpha$
    coincides and the
    Nichols algebra $\B(V)$ is the Nichols algebra associated to the Lie algebra
    $\g$ and hence of dimension $\prod_{\alpha\in \Phi^+}\ell_\alpha^{|\Phi^+|}$
    (note that we
    inspected every row in \cite{Heck09} and provided $\ord(q^2)>d_\alpha$ the
    Nichols algebras in question have indeed the claimed root system and hence
    the claimed dimension).
    \item For $\g=B_n,q=\pm i$ we defined
      $X=\{\alpha_n,\alpha_n+\alpha_{n-1},\ldots\}$. An easy calculation shows
      $$q^{(\alpha_n+\alpha_{n-1}+\cdots \alpha_k,\alpha_n+\alpha_{n-1}+\cdots
      \alpha_l)}=\begin{cases} q^2,& k=l \\ 1,& k\neq l \end{cases}$$
      Hence the Nichols algebra has a root system of type $A_1^{\times n}$ and
      dimension $2^n$.
    \item For $\g=C_n,q=\pm i$ we defined
      $X:=\{\alpha_1,\alpha_2,\ldots,\alpha_{n-1}+\alpha_n\}$. The braiding
      matrix of the first $n-1$ root vectors is that of $A_n\subset C_n$ and for
      the braiding with the last root vector $E_{\alpha_{n-1}+\alpha_n}$ we
      easily obtain 
      $$q^{(\alpha_{n-2},\alpha_{n}+\alpha_{n-1})}=q^{-1}\qquad
      q^{(\alpha_{n-1},\alpha_{n-1}+\alpha_n)}=q^{0}\qquad
      q^{(\alpha_{n-1}+\alpha_n,\alpha_{n-1}+\alpha_n)}=q^{2}$$
      Hence the Nichols algebra has a root system of type $D_n$ (with
      $\alpha_{n-2}$ the center node) and dimension $2^{|\Phi^+|}=2^{n(n-1)}$.
    \item For $\g=F_4,q=\pm i$ we
      defined $X=\{\alpha_{4},\alpha_3,\alpha_3+\alpha_2,
      \alpha_{3}+\alpha_2+\alpha_1\}$. We explicitly calculate the
      braiding matrix of these root vectors:
      $$\begin{pmatrix}
	  q^{2} & q^{-1} & q^{-1} & q^{-1} \\
	  q^{-1} & q^2 & 1 & 1  \\
	  q^{-1} & 1 & q^2 & 1  \\
	  q^{-1} & 1 & 1 &  q^2 
        \end{pmatrix}$$
    Hence the Nichols algebra has a root system of type $D_4$ (with
    $\alpha_4$ the center node) and dimension  $2^{12}$. It extends the root
    system $A_1^{\times 3}$ for $B_3$ (generated by $\alpha_3,\alpha_3+\alpha_2,
    \alpha_{3}+\alpha_2+\alpha_1$) as well as  the root system $A_3$
    for  $C_3$ (generated by $\alpha_4,\alpha_3,\alpha_3+\alpha_2$).
    \item For $\g=G_2,q=\sqrt[3]{1},\sqrt[6]{1}$ we defined    
    $X=\{\alpha_1,\alpha_1+\alpha_2\}$. We explicitly calculate the braiding
    matrix of these root vectors:
    $$\begin{pmatrix}
	q^2 & q^{-1}\\
	q^{-1} & q^2 \\
        \end{pmatrix}$$
    Hence the Nichols algebra has a root system $A_2$ and dimension $3^3$.
    \item For the exotic case $\g=G_2,q=\pm i$ we defined 
    $X=\{\alpha_2,\alpha_1,2\alpha_{1}+\alpha_2\}$. We explicitly calculate the
    braiding matrix of these elements:
    $$\begin{pmatrix}
	q^6 & q^{-3} & 1 \\
	q^{-3} & q^{2} & q \\
	1 & q & q^{2}
      \end{pmatrix}\stackrel{q^4=1}{=}
      \begin{pmatrix}
	q^2 & -q^{-1} & 1 \\
	-q^{-1} & q^{2} & -q^{-1} \\
	1 & -q^{-1} & q^{2}
      \end{pmatrix}
      =\begin{pmatrix}
	\bar{q}^2 & \bar{q}^{-1} & 1 \\
	\bar{q}^{-1} & \bar{q}^{2} & \bar{q}^{-1} \\
	1 & \bar{q}^{-1} & \bar{q}^{2}
      \end{pmatrix}$$
      Hence the Nichols algebra has a root system $A_3$ and dimension $2^3$.
      However, the braiding matrix is not the standard braiding matrix from
      $u_q(A_3)$, rather for then small quantum group  $u_{\bar{q}}(A_3)$
      associated to the respective other choice of a primitive fourth root of
      unity, which is complex conjugate.
  \end{enumerate}
  \item Let $H$ denote the subalgebra of $u_q^\L$ generated by the chosen
  primitive root vectors $E_\alpha,\alpha\in X$, which span a vector space $V$.
  In all cases except the exotic case the $\N^n$-grading of these generators is
  linearly independent, hence $H$ is a $\N$-graded algebra. 
  By the universal property of the Nichols algebra we thus have a surjection
  $H\to\B(V)$. To finally show equality $u_q^\L\cong \B(V)$ we use Lusztig's
  PBW-basis of $u_q^{\L,+}$ in Theorem \nref{thm_Lusztiguq} to show in each
  case $\dim(u_q^{\L,+})=\dim(\B(V))$:
  \begin{enumerate}[i)]
  \item In the trivial case the set of all roots with $\ell_\alpha>1$ is empty,
  hence the dimension of $u_q^{\L,+}$ is $1$. In the generic case all roots
  have coinciding $\ell_\alpha>1$ hence the dimension of $u_q^{\L,+}$ is
  $\ell_\alpha^{|\Phi^+|}$.
  \item For all other cases except the exotic case $G_2,q=\pm i$ the set of all
  roots with $\ell_\alpha>1$ is precisely the set of short roots and all short
  roots fulfill $\ell_\alpha=2$ (resp. $=3$ for
  $G_2,q=\sqrt[3]{1},\sqrt[6]{2}$).  Hence the dimension of $u_q^{\L,+}$ is
  $2^N$  (resp. $3^N$) with $N$ the  number of positive short roots, i.e. $n,
  n(n-1),12,3$ for $B_n,C_n,F_4,G_2$.
  \item 
  For the exotic case $G_2,q=\pm i$ we have for long and short roots
  $\ell_\alpha=2$. Since $G_2$ has $6$ positive roots, the dimension of
  $u_q^{\L,+}$ is $2^6$. It is quite remarkable that this coincides with the
  dimension $2^6$ of the Nichols algebra of type $A_3,q=\pm i$. Note that
  a-priori it is not clear $u_q^{\L,+}$ is a graded Hopf algebra. This only
  follows after inspecting the relations between $E_2,E_1$ and $E_1,E_{112}$
  and $E_{1112},E_2$ and $E_{12},E_{112}$ in \cite{Lusz90b} Sec. 5.2, which are
  all graded with respect to $E_2,E_1,E_{112}$ having degree $1$ (except the
  $[2]E_{112}$-term for $E_1,E_{12}$, which is zero for $q=\pm i$).
  \end{enumerate}
  We notice that in each case the dimension of $\B(V)$ calculated in b)
  coincides with the dimension of $u_q^{\L,+}$ by Lusztig's PBW-basis obtained
  in c). Hence the two are isomorphic which concludes the proof.
\end{enumerate}
\end{proof}

\newpage
\section{A short exact sequence}\label{sec_shortexactsequence}
\enlargethispage{.3cm}
Lusztig has in \cite{Lusz90b} Thm 8.10. discovered a remarkable homomorphism 
from $U_q^\L(\g)$ to the ordinary universal enveloping Hopf algebra $U(\g)$ 
$$U_q^\L(\g,\Lambda)\xrightarrow{Frob}U(\g)$$
whenever $\ell$ is odd and for $\g=G_2$ not divisible by $3$. He called it
\emph{Frobenius homomorphism} to emphasize it should be viewed as a
``lift'' of the Frobenius homomorphism in characteristic
$\ell$ to the quantum group in characteristic $0$.\\

The following is a more systematic construction, using the techniques of
Nichols algebras and generalizes to the situation of small prime
divisors (which has been excluded by Lusztig and throughout the following
literature, note however Lusztig's book \cite{Lusz94} Thm. 35.1.9
for large order but small prime divisors). First, we prove that
$u_q^\L(\g,\Lambda)$ is a normal Hopf subalgebra in $U_q^\L(\g,\Lambda)$, this
relies crucially on the explicit description of the assumed kernel
$u_q^\L(\g,\Lambda)$ by Theorem \nref{thm_smalluq}. Then we
form the quotient Hopf algebra, the quotient is then the quantum Frobenius
homomorphism. Finally we inspect the quotient and prove it is (close to) a
universal enveloping of some Lie algebra $\g^{(\ell)}$.

\begin{theorem}\label{thm_main}
  Depending on $\g$ and $\ell$ we have the following exact sequences of Hopf
  algebras in the category of $\Lambda$-Yetter-Drinfel'd modules:
  $$u_q(\g^{(0)},\Lambda)^+\xrightarrow{\;\subset\;} U_q^\L(\g,
\Lambda)^+\xrightarrow{\;Frob\;}U(\g^{(\ell)})^+$$
\begin{center}
\begin{tabular}{l|ll|l|ll}
& $\g\qquad$ & $\ell=\ord(q)$ & $\g^{(0)}\quad$ &
$\g^{(\ell)}\quad$ & is braided for\\
\hline\hline
\textnormal{Trivial cases:}
& all & $\ell=1$ & $0$ & $\g$ & no \\
& all & $\ell=2$ & $0$ & $\g$ & $ADE_{n\geq 2},C_{n\geq 3},F_4,G_2$ \\
\cline{2-6}
\multirow{4}{*}{\textnormal{Generic cases:}} 
& $ADE$ & $\ell\neq 1,2$ & $\g$ & $\g$ & $\ell=2\mod 4,n\geq 2$ \\
&$B_n$ & $4\nmid \ell\neq 1,2$ & $B_n$ & $B_n$ & no \\
&$C_n$ & $4\nmid \ell\neq 1,2$ & $C_n$ & $C_n$ & $\ell=2\mod 4,n\geq 3$ \\
&$F_4$ & $4\nmid \ell\neq 1,2$ & $F_4$ & $F_4$ & $\ell=2\mod4$\\ 
&$G_2$ & $3\nmid\ell\neq 1,2,4$ & $G_2$ & $G_2$ & $\ell=2\mod 4$\\
\cline{2-6}
\multirow{8}{*}{\textnormal{Duality cases:}$\quad$}
& $B_n$ & $4|\ell\neq 4$ & $B_n$ & $C_n$ & $\ell=4\mod 8,n\geq 3$ \\
&& $\ell=4$ & $A_1^{\times n}$ & $C_n$ &  $n\geq 3$\\      
& $C_n$ & $4|\ell\neq 4$ & $C_n$ & $B_n$ &  no \\
&& $\ell=4$ & $D_n$ & $B_n$ & no \\
& $F_4$ & $4|\ell\neq 4$ & $F_4$ & $F_4$ & $\ell=4\mod 8$\\
&& $\ell=4$ & $D_4$ & $F_4$ & yes \\ 
& $G_2$ & $3|\ell\neq 3,6$ & $G_2$ & $G_2$ & $\ell=2\mod 4$\\
&& $\ell=3,6$ & $A_2$ & $G_2$ & $\ell=6$\\
\cline{2-6}
\textnormal{Exotic case:}
& $G_2$ & $\ell=4$ & $A_3$ & $G_2$ & no \\ 
\end{tabular}
\end{center}
\end{theorem}
The author would be very interested to
obtain a similar list for affine Lie algebras as well other Nichols
algebra extensions (see Problems \nref{prob_affine} and
\nref{prob_extension}).\\
\begin{proof}
  The rest of this article is
  devoted to prove this theorem as follows: 
 \begin{itemize}
  \item The structure of $u_q^\L(\g,\Lambda)^+$ in the table column $\g^{(0)}$
  was already determined in Theorem  \nref{thm_smalluq}.
  \item By Theorem \nref{thm_pairs} for the root systems in question all pairs
    of roots can be reflected simultaneously into some parabolic subgroup of
    rank $2$. This will be excessively used in the following two steps to
    reduce all calculations to rank $2$.
  \item In Lemma \nref{lm_normal} we prove $u_q^\L(\g,\Lambda)$ is a normal
    Hopf subalgebra of $U_q^\L(\g,\Lambda)^+$ in the category of 
    $\Lambda$-Yetter-Drinfel'd modules. By simultaneous reflection we will only
    have to check rank $2$, then we use the convenient generator set for
    $u_q^\L(\g,\Lambda)$. Note that for the exotic case $G_2,q=\pm 1$ is would
    not suffice to check the action on simple root vectors (as one might do),
    since there is an additional algebra generator $E_{112}$.
  \item  Then we will then consider the quotient of
    $U_q^\L(\g,\Lambda)^+$ by the normal Hopf subalgebra $u_q^\L(\g,\Lambda)$.
    We show in Lemma \nref{lm_gell} that the quotient is generated by primitive
    elements $E_{\alpha_i}^{(\ell_{\alpha_i})}$ with the (possibly braided)
    commutator structure as prescribed in the table column $\g^{(\ell)}$. We
    will do so again by trick a) to reduce to rank $2$ and check the
    isomorphisms explicitly. Note that the braiding corresponds to even lattices
    in the Lie algebra Lemma \nref{lm_ellLattice} and that (independently) the
    dual root system is formed from the exceptions in Lemma \nref{lm_commute}. 
 \end{itemize}
Note that the cases with $\ord(q^2)>d_\alpha$ have been verified in \cite{Lusz94}.
\end{proof}

\begin{remark}
  Recently Angiono has in \cite{An14} characterized (dually) the Borel part of
  the Kac-Procesi-DeConcini-Form $U_q^\K(\g)^+$ purely in terms of
  so-called distinguished Pre-Nichols algebra in the braided category of
  $\Lambda$-Yetter-Drinfel'd Modules. These algebras are much more
  general and all come with a version of a Frobenius homomorphism.
\end{remark}

\subsection{Orbits of pairs in root systems}

We will start with a technical theorem (which may be known) that should be in
general helpful for quantum groups of high rank $\g$ by reducing issues to
rank $2$. It has been already observed by Lusztig in \cite{Lusz90b} Sec. 3.6 for
the simply-laced case as part of an explicit description of all roots by
diagrams $\Gamma_i$. 

\begin{theorem}\label{thm_pairs}\footnote{I am very thankful to the 
	referee for pointing out the much shorter and conceptual proof given here for a), which also 
	works for larger $n$-tuples of roots.}~
\begin{enumerate}[a)]
  \item For any pair of roots $\alpha\neq \pm\beta$ there is a Weyl
  group element mapping $\alpha,\beta$ \emph{simultaneously} into a rank $2$
  parabolic subsystem $\langle\alpha_i,\alpha_j\rangle$.\\
	More generally, for any set $A$ of roots 
	of order $|A|<\rank(\g)$ there is a Weyl group element mapping $A$ simultaneously into a 
	parabolic subsystem of rank $|A|$.
  \item The type of the rank $2$ subsystem $\langle\alpha_i,\alpha_j\rangle$
  (including lengths) is uniquely determined by $\alpha,\beta$ in a) and all such
  subsystems are in a single Weyl group orbit, except three parabolic $A_1\times
  A_1$ in $D_4$ interchanged
  by the triality diagram automorphism as well as two orbits for $D_{n},n\geq
  5$, namely the parabolic $A_1\times A_1$ consisting of the two tiny
  legs $\alpha_n,\alpha_{n-1}$ and all remaining parabolic $A_1\times A_1$.
  \item Unordered pairs of roots in a Lie algebra of rank $2$ are classified by
  length and angle up to action of the Weyl group.
\end{enumerate}
Altogether, Weyl orbits of unordered pairs of roots are in bijective
correspondence to types of rank $2$ parabolic subsystems, lengths and angle
with the mentioned exceptions.
\end{theorem}
Before we proceed to the proof we give examples how this theorem works and
fails:
\begin{example}
  For $B_n,n\geq 4$ the possible types of
  parabolic subsystems of rank $2$ are  $B_2,A_2^{long}, A_1^{long}\times
  A_1^{long}, A_1^{short}\times A_1^{long}$. In the subsystem $B_2$ there are
  four orbits (classified by their lengths and angle), in $A_2$ are two orbits
  (classified by their lengths and angle)  and in the others each one orbit.
  Hence the Theorem returns $8$ orbits of pairs, each with an explicit
  representative inside a rank  $2$ parabolic subgroup.
\end{example}
\begin{example}\label{exm_pair}
  We show that the theorem fails for the exception $D_4$:
  The Weyl group acts transitively, so the one-point stabilizer (fixing
  $\alpha$) has order $|W|/|\Phi|=8$. The number of $\beta$ in
  an orbit of pairs $(\alpha,\beta)$ has to divide this order, the quotient
  being the order of the stabilizer of the (ordered) pair. For $\g=D_4$ the positive roots 
	orthogonal to a given $\alpha$ (say the highest root $\omega$) are three simple roots not in the
  center $\alpha_1,\alpha_2,\alpha_3$. The three pairs $\{\omega,\alpha_i\}$
  are interchanged by the triality diagram automorphism, but since $3\nmid 8$
  they have to belong to different Weyl group orbits. After reflection, these
  three orbits can be recognized as the three parabolic subsystems 
  $\{\alpha_i,\alpha_j\}$ of type $A_1\times A_1$, hence statement a) holds, but not b).
\end{example}

\begin{proof}[Proof of Theorem \nref{thm_pairs}]
For $\g$ of rank $2$ the statements a) and b) are trivial, the statement c)
follows by explicit inspection. Note that in $A_2$ there are two orbits of
\emph{ordered} pairs interchanged by the diagram automorphism.
\begin{enumerate}[a)]
	\item The following proof shows in fact that any set of roots $A\subset \Phi^+$ of order $k<n=:
	\rank(\g)$ can be simultaneously reflected into a parabolic subsystem of rank $k$. By induction 
	it suffices to show that we can reflect $A$ into a parabolic subsystem of rank $<n$, then we my 
	proceed until $k=n$. Let $v$ be 
	a vector in 
	$\Phi\otimes \mathbb{R}$, which is of dimension $n$, such that $v\perp A$. Consider the set 
	$$M_v:=\{\gamma\in\Phi^+\mid (v,\gamma)<0\}$$ 
	If $M_v\neq\varnothing$ then there exists at least one simple root $\alpha_i\in M_v$ 
	(otherwise, being positive linear combinations, no positive root could be in $M_v$). If we apply 
	a reflection $s_i$, then since $(s_i\alpha,s_i\beta)=(\alpha,\beta)$: 
	$$M_{s_iv}=\{\gamma\in \Phi^+\mid (s_iv,\gamma)<0\}
	=\Phi^+\cap s_i(M_v)=s_i\left(M_v\backslash\{\alpha_i\}\right)$$
	So in terms of cardinality $|M_{s_iv}|=|M_v|-1$. Hence by successive reflection one can find $v
	\perp A$ with $M_v=\varnothing$, especially all 
	$(v,\alpha_i)\geq 0$. But this 
	implies that for any $v\perp \alpha=\sum_i n_i\alpha_i$ we have $v\perp\alpha_i$ for all $n_i
	\neq 0$. Hence $v^\perp\cap \Phi$ is a parabolic subsystem of rank $<n$ and by construction it 
	contains $A$.
  \item It is clear that the type of the rank $2$ parabolic subsystem (including
  lengths) is uniquely determined by projecting to the  vector subspace (note
  the statement is more trivially true in most cases by  angle and length, but
  e.g. orthogonal long roots in $B_2$ are distinguished  from orthogonal roots
  in $A_1^{long}\times A_1^{long}$, since in the  former there exists a
  $2\gamma=\alpha+\beta$).\\

  We wish to show that all rank $2$ parabolic subsystems are in a single
  Weyl orbit with the exception that the three resp. two parabolic $A_1\times A_1$
  in $D_4$ resp. $D_n$ interchanged by diagram automorphism are in different
  orbits. This is done by inspecting every case of $\g$ explicitly (and mostly
  reduce to previous cases):
  \begin{enumerate}[i)]
   \item Let $\g=A_n,n\geq 3$, then we have rank $2$ parabolic subsystems of
    type $A_2$ and $A_1\times A_1$. It is easy to shift a subsystem
    $\{\alpha_i,\alpha_{i+1}\}$ to $\{\alpha_{i-1},\alpha_{i}\}$ by reflecting on
    $\alpha_{i-1},\alpha_i,\alpha_{i+1}$. Similarly, one can easily give
    direct expressions that shift any $\alpha_i,\alpha_j,|i-k|>1$ to any
    other such pairs by shifting each one separately (more abstractly spoken,
    by using the transitivity of the Weyl group of a suitable subsystem).
   \item Let $\g=B_n,n\geq 3$ (resp $C_n$ by duality) then we have rank $2$
    parabolic subsystems of type $B_2$ and $A_2^{long}$ and for $n\geq 4$ of
    type $A_1^{long}\times A_1^{long}$. But the parabolic subsystem $B_2$ is
    unique and all the other subsystems lay in the parabolic subsystem of type
    $A_{n-1}^{long}$, for which the claim has been shown in a).
   \item Let $\g=D_n,\;n\geq 4$. Any two parabolic subsystems of type $A_2$ lay in
    a common subsystem of type $A_{n-1}$, hence by i) we are finished. For the
    parabolic case of type $A_1\times A_1$ we have already clarified the
    situation for $D_4$ in Example \nref{exm_pair}, so let's assume $n\geq 5$:
    By transitivity of the Weyl group, fix $\alpha=\alpha_1$ the outmost vertex, then the roots 
		orthogonal to $\alpha_1$ are of the form $k\cdot\alpha_1+2k\cdot\alpha_2+\cdots$. The roots 
		in the class $k=0$ are by definition the parabolic 
		subsystem $D_{n-2}$ and explicit inspection of the root system shows the only element in the 
		class $k=1$ is the highest root and there are no roots for $k>1$. All roots in the first 
		class can by the mapped to each other by the Weyl group of
    $D_{n-2}$  without affecting $\alpha_1$. We only have to show that this is
    not possibly for the highest root as well and we argue as follows: The Weyl
    group of $D_n$ has order $2^{n-1}n!$ and the root system has $2n(n-1)$   
    roots. Hence the one-point-stabilizer has double order as the contained
    Weyl group of $D_{n-2}$ and additionally contains the reflection on the
    highest root $\omega$ of $D_n$, which accounts for the entire stabilizer.
    But none of these elements can map $\omega$ to something else than
    $\pm\omega$, which shows it forms a single orbit. It is easy to
    see that the pair $\alpha_1,\omega$ has to be in the same orbit as the
    simple roots $\alpha_n,\alpha_{n-1}$ at the tiny legs of $D_n$.
   \item Let $\g=E_n,\;n=6,7,8$. Any two parabolic subsystems are in a
    parabolic subsystem of type $A_k$ hence by i) we are finished. For a
    parabolic subsystem $\alpha_i,\alpha_j$ of type $A_1\times A_1$ we have
    three cases: Either both $\alpha_i,\alpha_j$ are not the
    tiny leg $\alpha_2$, or say $\alpha_i$ is the tiny leg and $\alpha_j$ is on
    either of the other legs. All subsystems in each class are in a parabolic
    subsystem of type $A_{n-1},A_{n-2},A_4$, hence each case for an orbit by
    i). It remains to map a representative of each case to
    another case (which was not possible for $D_n$): We move $\alpha_j$ (not
    the tiny leg) to an outmost vertex without affecting $\alpha_i$ using a
    parabolic subsystem of type $A_k$. Then we may change $\alpha_i$ to the
    other case by using the $A_3$ subsystem around the center node.
   \item Let $\g=F_4$, then we have unique rank $2$ parabolic subsystems of type
    $A_2^{long}$, $A_2^{short}$, $B_2=C_2$ and three parabolic
    subsystems of type $A_1^{long}\times A_1^{short}$, namely
    $\{\alpha_1,\alpha_3\},\{\alpha_1,\alpha_4\},$ $\{\alpha_2,\alpha_4\}$.
    These three parabolic subsystems can be reflected to each other exactly as
    in the $A_n$ case a), explicitly:
    \begin{align*}
      &\{\alpha_1,\alpha_3\}
      \stackrel{\alpha_4}{\mapsto}\{\alpha_1,\alpha_3+\alpha_4\}
      \stackrel{\alpha_3}{\mapsto}\{\alpha_1,\alpha_4\}\\
      &\{\alpha_1,\alpha_4\}
      \stackrel{\alpha_2}{\mapsto}\{\alpha_1+\alpha_2,\alpha_4\}
      \stackrel{\alpha_1}{\mapsto}\{\alpha_2,\alpha_4\}
    \end{align*}
  \end{enumerate}
  \item We finally convince ourselves for $A_2,B_2,G_2$ that lengths and angle
  of $\{\alpha, \beta\}$ classify the pairs: Fix one root by transitivity, then
  there is a unique choice $\pm\beta$, for $B_2,G_2$ we can reflect on a root
  orthogonal to $\alpha$, for $A_2$ we can reflect $\{\alpha,-\beta\}$ to
  $\{\beta,\alpha\}$. Note that the \emph{ordered} pairs $\alpha_1,\alpha_2$
  and $\alpha_2,\alpha_1$ are not in the same orbit.
\end{enumerate}
\end{proof}

\subsection{Adjoint action on the small quantum group}

\begin{lemma}\label{lm_normal}
  The Hopf subalgebra $u_q^\L(\g,\Lambda)^+\subset U_q^\L(\g,\Lambda)^+$ from
  Definition \nref{def_uqL} is a normal Hopf subalgebra, i.e. stable under the
  adjoint action. We explicitly give the respective (skew-)derivations in the
  degenerate cases of Theorem \nref{thm_smalluq}.
\end{lemma}
\begin{proof}
  For $\ord(q^2)>d_\alpha$ Lusztig obtained $u_q^\L(\g,\Lambda)^+$ in
  \cite{Lusz94} Thm. 35.1.9 as kernel of his explicit Frobenius homomorphism,
  so it remains to check the degenerate cases in Theorem \nref{thm_smalluq}.
  Note that our proof works in other cases as well.\\

  To calculate $\ad_{E_{\alpha}^{(\ell_\alpha)}}(E_{\beta})$ we may invoke 
Theorem \nref{thm_pairs}
to find a Weyl group element that maps $\alpha,\beta$  to a rank $2$ parabolic
subsystem. It hence suffices to check the cases of rank $2$. Using another
reflection we may assume $\alpha$ to be a simple  root. Moreover it suffices to
check normality on a set of generators for   $u_q^\L(\g,\Lambda)^+$ as
explicitly given in Theorem \nref{thm_smalluq}. From
 $$\Delta(E_{\alpha_i}^{(k)})=\sum_{b=0}^kq^{d_{\alpha_i}b(k-b)}
  E_{\alpha_i}^{(N-b) } \otimes E_{\alpha_i}^{(b)}$$ 
  we see that the adjoint action of any $E_{\alpha_i}^{(\ell_i)}$ is the sum of
  the braided commutator $\delta_i:=[E_{\alpha_i}^{(\ell_i)},\_]$
  and lower terms $E_{\alpha_i}^{(k)},k<\ell_i$ that are by definition
completely
in $u_q^{\l,+}$. These $\delta_i$ are (sometimes symmetrically braided)
derivations, and have been introduced by Lusztig, who states in
\cite{Lusz90b} Lm. 8.5 that they preserve
$u_q^\L$ (for $2\nmid \ell$ and $3\nmid \ell$ for $\g=G_2$). We will restrict
ourselves to computing $\delta_i$ in each degenerate case, since they capture
the nontrivial (non-inner) part of the adjoint action on $u_q^{\L,+}$ and are
invariant under reflection (which is an algebra map). Thus we can prove
normality:\\

  We calculate $\delta_i=[E_{\alpha_i}^{(\ell_i)},\_]$ on each algebra
  generator of $u_q^{\L,+}$ as given in Theorem \nref{thm_smalluq}. Whenever
  $\alpha_i+\beta\not\in\Phi^+$ we have $\delta_i(E_\beta)=0$ except possibly
  the four cases in Lemma \nref{lm_commute}
  $$B_2:\;(\alpha_{112},\alpha_2)\qquad
    G_2:\;(\alpha_{11122},\alpha_2),
    (\alpha_{112},\alpha_2),(\alpha_1,\alpha_{11122})$$
  where the last two are in a Weyl orbit. For the remaining generators we
  proceed  case-by-case using the commutation formulae in \cite{Lusz90b}
  Sec. 5:
\begin{enumerate}[a)]
  \item For $\g=A_1\times A_1$ there is no pair $\alpha+\beta\in\Phi$.
For $q=-1$ we  have $[E_\alpha,E_\beta]=-[E_\alpha,E_\beta]$ hence both
$\delta_1,\delta_2$   vanish.
  \item For $\g=A_2,q=\pm i$ (parabolic in $C_n,F_4,q=\pm i$) we have
  $\ell_1=\ell_2=2$ and we check the two pairs with $\alpha_i+\beta\in\Phi$:
  \begin{align*}
    \delta_2(E_1)
    &=E_2^{(\ell_2)}E_2-q^{-\ell_2}E_2E_1^{(\ell_2)}\\
    &=E_2^{(\ell_2)}E_2-q^{-\ell_2}
   \sum_{\substack{r+s=\ell_1\\s+t=1}}q^{tr+s}E_2^{(r)}E_{12}^{(s)}E_1^{(t)}\\
    &=E_2^{(\ell_2)}E_2-q^{-\ell_2+\ell_1}E_2^{(\ell_2)}E_2	
    -q^{-\ell_2+1}E_2^{(\ell_2-1)}E_{12}
    \stackrel{\ell_2=2}{=}qE_2E_{12}\\
    \delta_1(E_2)
    &=E_1^{(\ell_1)}E_2-q^{-\ell_1}E_2E_1^{(\ell_1)}\\
    &=\sum_{\substack{r+s=1\\s+t=\ell_1}}q^{tr+s}E_2^{(r)}E_{12}^{(s)}E_1^{(t)}
    -q^{-\ell_1}E_2E_1^{(\ell_1)}\\
    &=(q^{\ell_1}-q^{-\ell_1})E_1^{(\ell_1)}
    +qE_{12}E_1^{(\ell_1-1)}
    \stackrel{\ell_1=2}{=}qE_{12}E_1
  \end{align*}
  \item For $\g=B_2,q=\pm i$ the subalgebra $u_q^{\L,+}$ (in fact of type
  $A_1\times A_1$) is generated by the short root vectors $E_1,E_{12}$ and
  $\ell_1=2,\ell_2=1$ so we have to check the following two pairs with
  $\alpha_i+\beta\in\Phi$:
  \begin{align*}
    \delta_2(E_1)
    &=E_2E_1-q^{-2}E_1E_2\\
    &=E_2E_1-q^{-2}\left(q^2E_2E_1+q^2E_{12}\right)=-E_{12}\\
    \delta_1(E_{12})
    &=E_1^{(2)}E_{12}-E_{12}E_1^{(2)}\\
    &=\sum_{\substack{r,s,t\geq 0\\r+s=1\\s+t=2}}
    q^{-sr-st+s} \left(\prod_{i=1}^s\left(q^{2i}+1\right)\right)
    E_{12}^{(r)}E_{112}^{(s)}E_{1}^{(t)}-E_{12}E_1^{(2)}\\
    &=(q^{2}+1)E_{112}E_{1}\stackrel{q=\pm i}{=}0\\
  \end{align*}
  Note that $E_{112}$ would not have been in $u_q^{\L,+}$.
  \item For $\g=G_2,q=\sqrt[3]{1},\sqrt[6]{1}$ the
  subalgebra $u_q^{L,+}$ (in fact of type $A_2$) is generated by the short root
  vectors  $E_{\alpha_1},E_{\alpha_{12}}$ and we have
  $\ell_{1}=\ell_{12}=\ell_{112}=3,\ell_{2}=\ell_{1112}=\ell_{11122}=1$. We
  denote $\epsilon:=q^3=\pm 1$. We have to check two pairs
  $\delta_2(E_1)$ and $\delta_2(E_{12})$. Reflection on $\alpha_2$ maps
  $\alpha_2\leftrightarrow \alpha_{12}$ hence we may alternatively check
  $\delta_{12}(E_{2})$:
  \begin{align*}
    \delta_2(E_1)
    &=E_2E_1-q^{-3}E_1E_2
    \stackrel{(a2)}{=}E_2E_1-q^{-3}(q^3E_2E_1+q^3E_{12})=-E_{12}\\
    \delta_{12}(E_{2})
    &=E_{12}^{(3)}E_2-q^{-9}E_2E_{12}^{(3)}\\
    &\stackrel{(a3)}{=}E_{12}^{(3)}E_2-q^{-9}(q^3E_{12}^{(3)}E_1
    +q(q+q^{-1})E_{12}^{(2)}E_{112}+q^{-1}(q^2+1+q^{-2})E_{12}E_{11122})\\
    &\stackrel{{\mathrm ord}(q^2)=3}{=}-\epsilon qE_{12}^{(2)}E_{112}
  \end{align*}
  This is a a product of short root
  vectors $E_{\gamma}^{(k)},k<3$, so again in $u_q^{\L,+}$. Note that
  $E_{11122}$ would not have been in $u_q^{\L,+}$.
  \item For $\g=G_2,q=\pm i$ we have all $\ell_\alpha=2$ so all
  $E_\alpha^{(k)},k<2$ are in $u_q^{\L,+}$. However the subalgebra 
  $u_q^{\L,+}$  (in fact of type $A_3$)  is generated by the short root vectors 
  $E_1,E_{2},E_{112}$; this is why  having Theorem \nref{thm_smalluq} is crucial
  for the proof of this theorem. Thus we have to check the four pairs
  $\delta_1(E_2),\delta_2(E_1),\delta_1(E_{112})$
  as well as the exception in Lemma \nref{lm_commute} $\delta_2(E_{112})$. The
  first two pairs are easy:
  \begin{align*}
  \delta_1(E_2)
  &=E_1^{(2)}E_2-q^{-6}E_2E_1^{(2)}\\
  &\stackrel{(a6)}{=}q^6E_2E_1^{(2)}+q^5E_{12}E_1
  +q^4E_{112}-q^{-6}E_2E_1^{(2)}\stackrel{q=\pm i}{=}qE_{12}E_1+E_{112}\\
  \delta_2(E_1)
  &=E_2^{(2)}E_1-q^{-6}E_1E_2^{(2)}\\
  &\stackrel{(a2)}{=}E_2^{(2)}E_1
  -q^{-6}\left(q^6 E_2^{(2)}E_1+q^3E_2E_{12}\right)=-q^{-3}E_2E_{12}
  \end{align*}
  For the other two pair we have to proceed as follows: The  Weyl group elements
  $T_2T_1T_2$ and $T_2T_1$ both map $\alpha_{112}\mapsto\alpha_1$, the first
  one maps $\alpha_1\mapsto \alpha_{112}$ and the second one maps
  $\alpha_2\mapsto \alpha_{11122}$. We may hence alternatively calculate
  \begin{align*}
    \delta_{112}(E_1)
    &:=E_{112}^{(2)}E_1-q^2E_1E_{112}^{(2)}\\
    &\stackrel{(a4)}{=}E_{112}^{(2)}E_1
    -q^2\left(q^{-2}E_{112}^{(2)}E_1+q^{-3}(q^2+1+q^{-2})E_{112}E_{1112}\right)
    \stackrel{q=\pm i}{=}q^{-1}E_{112}E_{1112}\\
    \delta_{11122}(E_1)
    &:=E_{11122}^{(2)}E_1-E_1E_{11122}^{(2)}\\
    &\stackrel{(a7)}{=}E_{11122}^{(2)}E_1-E_{11122}^{(2)}E_1
    -q^{-4}(1-q^4)E_{11122} E_{112}^{(2)}
    \stackrel{q=\pm i}{=}0
  \end{align*}
  Note that $E_{112}^{(2)}$ would not have been in $u_q^{\L,+}$.
\end{enumerate}
\end{proof}

\subsection{Structure of the quotient}

We have proven in Lemma \nref{lm_normal} that the Hopf subalgebra
$u_q^{\L,+}\subset U_q^{\L,+}$ described in Theorem \nref{thm_smalluq} is
normal. We may hence consider the left ideal and two-sided coideal
$U_q^{\L,+}\ker_{\epsilon}(u_q^{\L,+})$ which is by normality a two-sided Hopf
ideal. We now form the Hopf algebra quotient:
\begin{definition} Define the following Hopf algebra in the
category of $\Lambda$-Yetter-Drinfel'd modules:
$$H:=U_q^{\L,+}/U_q^{\L,+}\ker_{\epsilon}(u_q^{\L,+})$$
\end{definition}

The goal of this section is to analyze $H$ and prove it is isomorphic to the
a Hopf algebra $U(\g^{(\ell)})^+$ with $\g^{(\ell)}$ the Lie algebra given for
each case in the statement of Theorem \nref{thm_main}.

\begin{lemma}
  $H$ has as vector space a PBW-basis consisting of monomials
  in $E_{\alpha}^{\ell_\alpha}$ for each positive root
  $\alpha$. 	
\end{lemma}
\begin{proof}
  This follows from the PBW-basis given in Lemma \nref{lm_SpecializationPBW}
  and the fact that by construction $u_q^{\L,+}$ has a PBW-basis consisting of
  monomials in $E_{\alpha}^{k},k<\ell_\alpha$. Note that by using
  Lusztig's PBW-basis of root vectors, we have complete control over the vector
  space $U_q^{\L,+}$, also in degenerate cases. The involved question addressed
  in this article is the Hopf algebra structure.
\end{proof}

We next address the question when $H$ is actually an ordinary Hopf algebra.
This is precisely the use of the lattice calculations in Lemma
\nref{lm_ellLattice} -- even lattices will correspond to properly (but
symmetrically) braided Hopf algebras and are marked as such for each case in the
statement of Theorem \nref{thm_main}. 

\begin{lemma}\label{lm_braided}
   In the following cases is $H$ an ordinary complex Hopf algebra:
  \begin{center}
    \begin{tabular}{ll}
     $A_n,D_n,E_6,E_7,E_8,G_2,n\geq 2$ & $\ell\neq2\mod 4$\\
     $B_n,n\geq 3$ & $\ell\neq4\mod 8$\\
     $C_n,n\geq 3$ & $\ell\neq2\mod 4$\\
     $F_4$ & $\ell\neq 2,4,6\mod 8$\\
    \end{tabular}
    \end{center}
    In the other cases $H$ is a Hopf algebra in a symmetrically braided category
    explicitly described in the proof.
\end{lemma}
\begin{proof}
   The braiding between the generators of the PBW-basis is
  \begin{align*}
    c(E_\alpha^{(\ell_\alpha)}\otimes E_\beta^{(\ell_\beta)})
    =q^{(\ell_\alpha\alpha,\ell_\beta\beta)}
    \cdot E_\beta^{(\ell_\beta)}\otimes E_\alpha^{(\ell_\alpha)}
  \end{align*}
  We have proven in Lemma \nref{lm_ellLattice} that except in the excluded cases
  we have $(\ell_\alpha\alpha,\ell_\beta\beta)\in\ell\Z$, hence the braiding is
  trivial. Note that in the excluded cases we have
  $(\ell_\alpha\alpha,\ell_\beta\beta)\in\frac{\ell}{2}\Z$, hence $c^2=1$
  and the braiding is still symmetric. Moreover we've generally shown that 
  $(\ell_\alpha\alpha,\ell_\alpha\alpha)\in\ell\Z$, hence the self-braiding is
  trivial and $U^+$ is a domain (i.e. no truncations).
\end{proof}

The most tedious part is now to verify that $H$ is indeed the asserted universal
enveloping algebra. Note that the theorem of Kostant-Cartier could easily be
applied, but there are two downsides: For one there seems to be no apparent
reason, why the simple root vectors generate the entire algebra (compare the
case $u_q^\L$), especially since we do not have $F$'s (see Problem
\nref{prob_uq}). Moreover, we do not know a-priori whether some
$H=U(\g^{(\ell),+})$ would actually be the positive part of some semisimple Lie
algebra $\g^{(\ell)}$, especially in the braided cases), hence calculating the
Cartan matrix would not suffice.\\

Rather, we shall in the following explicitly check all braided commutators and
verify they actually lead to the positive part of the Lie algebra
$\g^{(\ell)})$ given for each case in the statement of the Main Theorem
\nref{thm_main}. This becomes again feasible through the trick staged in
Theorem \nref{thm_pairs}, namely reflecting the relevant cases to a rank $2$
parabolic subsystem. There we calculate by hand, which is quite tedious for
$G_2$:

\begin{lemma}\label{lm_gell}
   There is an algebra isomorphism $H\cong U(\g^{(\ell,^+)})$ for some 
   $\g^{(\ell,^+)}$ as follows:
  \begin{itemize}
    \item Generic case: For $\ell_\alpha=\ell_\beta$ for long and short roots we
    have $\g^{(\ell,+)}\cong \g^+$. Note that the isomorphism typically picks
    up scalar factors for non-simple root vectors.
    \item Duality case: For $\ell_\alpha\neq \ell_\beta$ for long and short
    roots we have $\g^{(\ell,+)}\cong (\g^{\vee})^{+}$ for the dual root
    system. More precisely we wish to prove:
    \begin{itemize}
    \item For $4|\ell$ we have $B_n^{(\ell,+)}\cong C_n^+$ and
    $C_n^{(\ell,+)}\cong B_n^+$ and a nontrivial automorphism
    $F_4^{(\ell,+)}\cong F_4^+$. All three maps double short roots and hence
    interchange short and long roots.
    \item For $3|\ell$ we have a nontrivial automorphism $G_2^{(\ell,+)}\cong
    G_2^+$. The map triples short roots and hence interchanges short and
    long roots.
    \end{itemize}
  \end{itemize}
  Especially this shows that $H$ is generated as an algebra by simple root
  vectors $E_\alpha^{(\ell_\alpha)}$. Since these are primitive up to lower
  terms (which vanish in $H$), this also proves we have a Hopf algebra
  isomorphism.
\end{lemma}
\begin{proof}
  The case $\ord(q^2)>d_\alpha$ has been treated in \cite{Lusz94} Thm.
  35.1.9. Hence we again restrict ourselves to the degenerate cases, but note that our proof 
	works in general. 
  The degenerate cases with small orders of $q$ were already given in Theorem
  \nref{thm_smalluq} and are of type $B_n,C_n,F_4,G_2$. To calculate
  (symmetrically) braided  commutators
  $[E_\alpha^{(\ell_\alpha)},E_\alpha^{(\ell_\alpha)}]$ (and especially verify
  they are nonzero at the respective cases) we may hence  invoke Theorem
  \nref{thm_pairs}, which states that $\alpha,\beta$ can be  mapped
  simultaneously  to a rank $2$ parabolic subsystem of $\g$ and we may  demand
  that one root is even a simple root.\\

  We hence have to verify for the potential rank $2$ parabolic subsystems
  $A_1\times A_1,A_2,B_2,G_2$ that the commutators 
  $[E_{\alpha_i}^{(\ell_\alpha)},E_\beta^{(\ell_\beta)}]$ for roots 
  $\alpha_i,\beta\in \Phi(\g)^+$ are zero and nonzero in $\g^{(\ell)}$ in 
  agreement with the assumed isomorphism $H\cong U(\g^{\ell})^+$.
  We again excessively use Lusztig's 
  commutation formulae \cite{Lusz90b} Sec. 5 and since we calculate in the
  quotient $H$, all terms $E_\alpha^{(k)},k<\ell_\alpha$ all zero.
  \begin{enumerate}[a)]
  \item First recall from Lemma \nref{lm_commute} that whenever
  $\alpha_i+\beta\not\in\Phi(\g)$ then $E_{\alpha_i}^{(k)},E_{\beta}^{(k')}$
  braided commute with the following exceptions
  $$B_2:\;(\alpha_{112},\alpha_2)\qquad
    G_2:\;(\alpha_{11122},\alpha_2),
    (\alpha_{112},\alpha_2),(\alpha_1,\alpha_{11122})$$
  (where the last two are in a Weyl orbit). These exceptions will be extremely
  crucial for the formation of the dual root system.
  \item $A_1\times A_1$ is trivial by a). 
  \item For $A_2$ (especially $\ell=2$) we wish to verify $H\cong
  U(\g^{(\ell)})^+$ for $\g^{(\ell)}=\g$. In view of a) we only have to check
  that the nontrivial commutator is indeed nonzero: 
  \begin{align*}
      E_1^{(k)}E_2^{(k')}
      &=\sum_{\substack{r+s=k'\\s+t=k}}q^{tr+s}E_2^{(r)}E_{12}^{(s)}E_1^{(t)}
   \end{align*}
  Since $A_2$ is simply-laced, all $\ell_\alpha=\ord(q^2)$ coincide. If we apply
  the commutation formula to $k=k'=\ell_\alpha$ all terms on the right
  hand side vanish in the quotient except $r=t=\ell_\alpha,s=0$ and
  $r=t=0,s=\ell_\alpha$, hence in the quotient:
    \begin{align*}
      E_1^{(\ell_\alpha)}E_2^{(\ell_\alpha)}
      &=q^{\ell_\alpha^2}E_2^{(\ell_\alpha)}E_1^{(\ell_\alpha)}
      +q^{\ell_\alpha}E_{12}^{(\ell_\alpha)}\\
      &\Rightarrow\quad \left[E_1^{(\ell_\alpha)}E_2^{(\ell_\alpha)}\right]
      =q^{\ell_\alpha}E_{12}^{(\ell_\alpha)}\quad\mbox{braided for }\ell=2\mod 4
   \end{align*}
  As proven already in Lemma \nref{lm_braided}, the commutator is
\emph{braided} whenever $\ell=2\mod 4$, since then
  $q^{\ell^2_\alpha}=q^{\ell\frac{\ell}{4}}=-1$. Note that the commutator
  calculation above is rather general and would in fact is able to treat
  all cases with equal $d_\alpha$ (simply-laced or not) -- this is roughly how
  Lusztig argues in \cite{Lusz90b} Lm.  8.5. for $2\nmid \ell$.
  \item For $\g=B_2,4|\ell$ (especially $\ell=4$) we wish to verify $H\cong
  U(\g^{(\ell)})^+$
  for $\g^{(\ell)}=B_2^\vee=C_2\cong B_2$ with the isomorphism doubling short
  roots and hence switching short and long roots. We have for short roots 
  $\ell_{1}=\ell_{12}=\frac{\ell}{2}$ and  for long roots
  $\ell_{2}=\ell_{112}=\frac{\ell}{4}$. We have to check the  commutators in the
  quotient $H$ for $\alpha_{1},\alpha_{2}$ and    $\alpha_{1},\alpha_{12}$ with
  $\alpha_i+\beta\in\Phi(\g)$ as well as the   exception
  $\alpha_{2},\alpha_{112}$ in a):
  \begin{align*}
  E_1^{(\frac{\ell}{2})}E_2^{(\frac{\ell}{4})}
  &=\sum_{\substack{r,s,t,u\geq 0 \\
  r+s+t=\frac{\ell}{4}\\s+2t+u=\frac{\ell}{2}}}
  q^{2ru+2rt+us+2s+2t} E_2^{(r)}E_{12}^{(s)}E_{112}^{(t)}E_1^{(u)}\\
  &=\underbrace{q^{2\frac{\ell}{4}\frac{\ell}{2}}
  E_2^{(\frac{\ell}{4})}E_1^{(\frac{\ell}{2})}}_{r=\frac{\ell}{4},s=0,t=0
  ,u =\frac{\ell}{2}}
  +\underbrace{q^{2\frac{\ell}{4}}E_{112}^{(\frac{\ell}{4})}}_{r=0,s=0,t=\frac{
  \ell } { 4 } ,u=0}\\
  \Rightarrow\quad\left[E_1^{(\frac{\ell}{2})}E_2^{(\frac{\ell}{4})}\right]
  &=-E_{112}^{(\frac{\ell}{4})}\\
  E_{112}^{(\frac{\ell}{4})}E_{2}^{(\frac{\ell}{4})}
  &=\sum_{\substack{r,s,t\geq 0\\r+s=\frac{\ell}{4}\\s+t=\frac{\ell}{4}}}
  q^{-2sr-2st+2s} \left(\prod_{i=1}^s\left(q^{2-4i}-1\right)\right)
  E_2^{(r)}E_{12}^{(2s)}E_{112}^{(t)}\\
  &=\underbrace{E_2^{(\frac{\ell}{4})}
  E_{112}^{(\frac{\ell}{4})}}_{r=\frac{\ell}{4},s=0,t =\frac{\ell}{2}}
  +\underbrace{\left(\prod_{i=1}^{\frac{\ell}{4}}\left(q^{2-4i}-1\right)\right)
  E_{12}^{(\frac{\ell}{2})}}_{r=0,s=\frac{\ell}{4},t=0}\\
\Rightarrow\quad \left[E_{112}^{(\frac{\ell}{4})},E_2^{(\frac{\ell}{4})}\right]
  &=\left(\prod_{i=1}^{\frac{\ell}{4}}\left(q^{2-4i}-1\right)\right)\cdot
  E_{12}^{(\frac{\ell}{2})}\quad \neq 0\qquad\mbox{since }-2>2-4i>2-\ell\\
  E_1^{(\frac{\ell}{2})}E_{12}^{(\frac{\ell}{2})}
  &=\sum_{\substack{r,s,t\geq 0 \\ r+s=\frac{\ell}{2}\\s+t=\frac{\ell}{2}}}
  q^{-sr-st+s}\left(\prod_{i=1}^s q^{2i}+1\right)
  E_{12}^{(r)}E_{112}^{(s)}E_1^{(t)}\\
  &=\underbrace{E_{12}^{(\frac{\ell}{2})}E_1^{(\frac{\ell}{2})}}
  _{r=\frac{\ell}{2},s=0,t=\frac{\ell}{2}}+
+\underbrace{q^{\frac{\ell}{2}}\left(\prod_{i=1}^\frac{\ell}{2}(q^{2i}+1)\right)
  E_{112}^{(\frac{\ell}{2})}}_{r=0, s=\frac { \ell } { 2 } , t=0 }\\ 
  \Rightarrow\quad\left[E_1^{(\frac{\ell}{2})},E_{12}^{(\frac{\ell}{2})}\right]
  &=0\qquad \mbox{since for $2|\ell$ the product vanishes for
  }i=\frac{\ell}{4}
\end{align*}
  We convince ourselves that this result agrees with the assumed root system
  $\g^{(\ell)}=B_2^\vee$ with $\alpha_1':=\frac{\ell}{2}\alpha_1$ now the long
  root and $\alpha_2':=\frac{\ell}{4}\alpha_2$ now the short root:
  \begin{itemize}
   \item The commutator for $\alpha'_1,\alpha'_2$ is nonzero and
    proportional to the root vector 
    $$\alpha'_1+\alpha'_2=\frac{\ell}{2}\alpha_1+\frac{\ell}{4}\alpha_2
    =\frac{\ell}{4}\alpha_{112}$$
   \item The commutator for $\alpha'_1+\alpha'_2,\alpha'_2$ involving the
    exceptional pair $\alpha_2+\alpha_{112}\not\in\Phi(\g)$ is nonzero
    and proportional to the root vector (with higher divided power)
    $$(\alpha'_1+\alpha'_2)+\alpha'_2
    =\frac{\ell}{2}\alpha_1+2\frac{\ell}{4}\alpha_2   
    =\frac{\ell}{2}\alpha_{12}$$
    \item The commutator for $\alpha'_1,\alpha'_1+2\alpha'_2$ is exceptionally
    zero in the case $4|\ell$ (even though $\alpha_1+\alpha_{12}\in\Phi$),
    which is in agreement with the assumed root system $B_2^\vee$.
    \item All other commutators are trivial by a) in agreement with $B_2^\vee$.
  \end{itemize}\noindent 
  We have hence verified $H\cong U(\g^{(\ell)})^+$ with
  $\g^{(\ell)}=B_2^\vee\cong B_2$ for $4|\ell$.
\end{enumerate}

For $\g=G_2$ we do not have the luxury of \cite{Lusz90b} Sec. 5.3 and instead
have to use Sec. 5.4. We restrict ourselves to the relevant cases $\ell=3,6$
(duality case, the most tedious) and $\ell=4$ (exotic case). We again use
the convention $\epsilon:=q^3=\pm1$. To reduce the number of commutator
calculations we can by reflection restrict ourselves to one representative per
Weyl group orbit of pairs $(\alpha,\beta)$. For $G_2$ these are classified by
angle and lengths.

\begin{enumerate}[a)]
\item[e)]
  For $\ell=3,6$ we wish to verify $H\cong U(\g^{(\ell)})^+$
  for $\g^{(\ell)}=G_2^\vee\cong G_2$ with the isomorphism tripling short
  roots and hence switching short and long roots. We have
  for short roots $\ell_{1}=\ell_{12}=\ell_{112}=3$ and for long roots
  $\ell_{2}=\ell_{1112}=\ell_{11122}=1$. Hence in the quotient $H$ all (left- or
  rightmost) $E_{\alpha}^{(k)}=0$ when $\alpha$ short and $0<k<3$.
  In view of  a) we have to check all pairs $\alpha_i+\beta\in\Phi^+(\g)$ as
  well as the three exceptions (of which one is a reflection of the other): We
  start with all pairs including a long root:
  \begin{align*}
    \left[E_{1}^{(3)},E_2\right]
	&=E_{1}^{(3)}E_2-q^{-9}E_2E_{1}^{(3)}\\
	&\stackrel{(a6)}{=}q^{9}E_2E_{1}^{(3)}+q^7E_{12}E_1^{(2)}
	+q^5E_{112}E_1+q^3E_{1112}-q^{-9}E_2E_{1}^{(3)}\\
	&=(\epsilon^3-\epsilon^{-3})E_2E_{1}^{(3)}+\epsilon E_{1112}
	=\epsilon E_{1112}\\
    \left[E_{1112},E_2\right]
	&=E_{1112}E_2-q^{-3}E_2E_{1112}\\
	&\stackrel{(a9)}{=} q^3E_2E_{1112}+(-q^4-q^2+1)E_{11122}
	+(q^2-q^4)E_{12}E_{112}-q^{-3}E_2E_{1112}\\
	&=(\epsilon-\epsilon^{-1})E_{1112}E_2+(-q^4-q^2+1)E_{11122} 
	=2E_{11122}\\
    \left[E_{112}^{(3)},E_2\right]
	&=E_{112}^{(3)}E_2-E_2E_{112}^{(3)}\\
	&\stackrel{(a8)}{=}E_2E_{112}^{(3)}
	+q^{-4}(q^{-3}-q^{3})E_{12}E_{11122}E_{112}
	+q^{-3}(q^{-6}-q^6)E_{11122}^{(2)} \\
	&+q^{-1}(q^{-2}-q^2)E_{12}^{(2)}E_{112}^{(2)}
	-E_2E_{112}^{(3)}\\
	&=\epsilon(\epsilon^{-2}-\epsilon^{2})E_{11122}^{(2)}=0\\
    \left[E_{11122},E_2\right]
	&=E_{11122}E_2-q^3E_2E_{11122}\\
	&\stackrel{(a7)}{=}q^{-3}E_2E_{11122}
	+q^{-3}(q^2-1)(q^4-1)E_{12}^{(3)}-q^3E_2E_{11122}\\
	&=(\epsilon^{-1}-\epsilon)E_2E_{11122}
	+\epsilon^{-1}(q^6-q^4-q^2+1)E_{12}^{(3)}
	=3\epsilon E_{12}^{(3)}
  \end{align*}
  The remaining pair of short roots are $\alpha_1+\alpha_{12}\in\Phi$ is more 
  work, due to $\ell_1=\ell_{12}=3$. We shall from  now on calculate in	
  $U_q^{\Q(q)}$ (to get rid of the divided power of $E_1$)  and successively 
  apply the commutation rule $(a3)$ for single powers  $E_{12}^{(k)}E_1$:
  \begin{align*}
  E_1^3E_{12}^{(3)}
  &=E_1^2\left(q^3E_{12}^{(3)}E_1+[2]qE_{12}^{(2)}E_{112}
    +[3]q^{-1}E_{12}E_{11122}\right)\\
  &=q^3E_1\left(q^3E_{12}^{(3)}E_1+[2]qE_{12}^{(2)}E_{112}
    +[3]q^{-1}E_{12}E_{11122}\right)E_1\\
  &+[2]qE_1\left(q^2E_{12}^{(2)}E_1+[2]qE_{12}E_{112}
    +[3]E_{11122}\right)E_{112}\\
  &+[3]q^{-1}E_1^2E_{12}E_{11122}\\
  &=q^6\left( q^3E_{12}^{(3)}E_1+[2]qE_{12}^{(2)}E_{112}
    +[3]q^{-1}E_{12}E_{11122} \right) E_1^2\\
  &+[2]q^4\left( q^2E_{12}^{(2)}E_1+[2]qE_{12}E_{112}
    +[3]E_{11122} \right)E_{112}E_1\\
  &+[2]q^3\left( q^2E_{12}^{(2)}E_1+[2]qE_{12}E_{112}
    +[3]E_{11122}\right)E_1E_{112}\\
  &+[2]^2q^2\left( qE_{12}E_1+[2]qE_{112}\right)E_{112}^2\\	
  &+[3]q^{-1}E_1^2E_{12}E_{11122}+[3][2]qE_1E_{11122}E_{112}
  \end{align*}
  After multiplying out we have four types of summands: The leading terms
  $q^9E_{12}^{(3)}E_1^3$ and $q^3[2]^3E_{112}^3$, several terms involving
  $[2][3]$ (say $X_1$), two terms involving  $[3]E_1^2$ (say $X_2$)
  and other terms involving $E_{12}$ or $E_{12}^{(2)}$, say  $Y_1$ and  $Y_2$.
  We shall further simplify the $Y_i$ and use  $E_1E_{112}=q^{-1}E_{112}E_1$:
  \begin{align*}
    Y_1&=[2]^2E_{12}\left(
      q^5E_{112}^2E_1
      +q^4E_{112}E_1E_{112}
      +q^3E_1E_{112}^2\right)\\
    &=[3][2]^2q^3E_{12}E_{112}^2E_1\\
    Y_2&=[2]E_{12}^{(2)}\left( 
      q^7E_{112}E_1^2 
      +q^6E_1E_{112}E_1
      +q^5E_1^2E_{112}\right)\\
    &=[3][2]q^5E_{12}^{(2)}E_{112}E_1^2
  \end{align*}
  so these terms $Y_1,Y_2$ can also be brought to a form involving $[3][2]$. If
  we now multiply the overall expression we derived for $E_1^3E_{12}^{(3)}$ by
  $\frac{1}{[3]!}$ and reinstate integral powers, we indeed
  find that all summands above are in the Lusztig integral form
  $U_q^{\Z[q,q^{-1}],\L}$ by themselves: We have the leading terms
  $q^9E_{12}^{(3)}E_1^{(3)}$ and $q^3[2]^3E_{112}^{(3)}$, terms $X_1',Y_1',Y_2'$
  where the present  $[3][2]=[3]!$ cancels and $X_2'$ (involving $E_1^2$) where
  $[3]$ cancels and we get $E_1^{(2)}$ from $E_1^2$.\\

  We may hence consider the above decomposition in the
  specialization $\U_q^{\L}$ and then in the quotient $H$:
  $$E_1^{(3)}E_{12}^{(3)}=q^9E_{12}^{(3)}E_1^{(3)}+q^3(q+q^{-1})^3E_{112}^{(3)}
  +X_1'+X_2'+Y_1'+Y_2'$$
  But in the quotient all monomials involving powers
  $E_\alpha^{(k)},E_\alpha^{k},k<3$ for  the short root vectors
  $E_1,E_{12},E_{112}$ as either leftmost or rightmost  factor vanish. We
  convince ourselves that such powers appear in every summand  of
  $X_0',X_1',Y_0',Y_1'$, which are hence zero in $H$. Hence we have finally
  proven for $q=\sqrt[3]{-1},\sqrt[6]{-1}$:
  $$\left[E_1^{(3)},E_{12}^{(3)}\right]=q^3[2]^3E_{112}^{(3)}
  =E_{112}^{(3)}\neq 0,\qquad\mbox{note }\ell_{112}=3$$
  We convince ourselves that our result agrees with the assumed root system
  $\g^{(\ell)}=G_2^\vee$ with $\alpha_1':=3\alpha_1$ now the long root
  and $\alpha_2':=\alpha_2$ now the short root:
  \begin{itemize}
   \item The commutator for $\alpha'_1,\alpha'_2$ is nonzero and
    proportional to the root vector 
    $$\alpha'_1+\alpha'_2=3\alpha_1+\alpha_2
    =\alpha_{1112}$$
   \item The commutator for $\alpha'_1+\alpha'_2,\alpha'_2$ is nonzero
    and proportional to the root vector 
    $$(\alpha'_1+\alpha'_2)+\alpha'_2
    =3\alpha_1+2\alpha_2
    =\alpha_{11122}$$
    \item The commutator for $\alpha'_1+2\alpha'_2,\alpha'_2$ involving the
    exceptional pair $\alpha_{11122}+\alpha_2\not\in\Phi(\g)$ is nonzero
    and proportional to the root vector (with higher divided power)
    $$(\alpha'_1+2\alpha'_2)+\alpha'_2
    =3\alpha_1+3\alpha_2
    =3\alpha_{12}$$
    \item The commutator for $\alpha_1,\alpha'_1+3\alpha'_2$ is
    nonzero and proportional to the root vector 
    $$\alpha'_1+(\alpha'_1+3\alpha'_2)
    =6\alpha_1+3\alpha_2
    =3\alpha_{112}$$ 
    \item The commutator for $2\alpha'_1+3\alpha'_2,\alpha'_2$ involving the
    exceptional pair $\alpha_{112}+\alpha_2\not\in\Phi(\g)$ is zero.
    \item All other commutators are trivial by a) in agreement with $G_2^\vee$.
  \end{itemize}
  We have hence verified $H\cong U(\g^{(\ell)})^+$ with
  $\g^{(\ell)}=G_2^\vee\cong
  G_2$ for $\ell=3,6$.
  \item[f)]  
  For $\ell=4$ we we wish to verify $H\cong
  U(\g^{(\ell)})^+$ for $\g^{(\ell)}=G_2$. We have $\ell_\alpha=2$ for all
  roots, so all $E_\alpha=0$ in the quotient $H$, which implies
  $E_\alpha^{(k)}=0$ for all $2\nmid k$. This present case hence works 
  analogously to b). However we have to exclusively calculate in $U_q^{\Q(q)}$
  as in the last  case of e). We only spell out one non-trivial case and the
  exceptional pair $(\alpha_2,\alpha_{112})$ in a):
  \begin{align*}
    E_1^2E_{112}^{(2)}
    &\stackrel{(a4)}{=}q^{-2}E_1E_{112}^{(2)}E_1+[3]q^{-3}E_1E_{112}E_{1112}\\
    &\stackrel{(a4)}{=}q^{-4}E_{112}^{(2)}E_1^2+[3]q^{-5}E_{112}E_{1112}E_1
    +[3]q^{-4}E_{112}E_1E_{1112}+[3]^2q^{-4}E_{1112}^2\\
    &=q^{-4}E_{112}^{(2)}E_1^2
    +[2][3]q^{-6}E_{112}E_{1112} E_1
    +[3]^2q^{-4}E_{1112}^2\\
    \Rightarrow \quad\left[ E_1^{(2)},E_{112}^{(2)} \right]
    &=[3]q^{-6}E_{112}E_{1112}E_1+[3]^2q^{-4}E_{1112}^{(2)}
    \stackrel{H}{=}[3]^2q^{-4}E_{1112}^{(2)}\stackrel{q=\pm i}
    {=}4E_{1112}^{(2)}\neq 0\\
E_{112}^{(2)}E_2^2
&\stackrel{(a8)}{=}E_2E_{112}^{(2)}E_2
+q^{-2}(q^{-3}-q^3)E_{12}E_{11122}E_2-[2](q^{-1}-q)E_{12}^{(2)}E_{112}E_2\\
&\stackrel{(a8)}{=}E_2^2E_{112}^{(2)}
+q^{-2}(q^{-3}-q^3)E_2E_{12}E_{11122}-[2](q^{-1}-q)E_2E_{12}^{(2)}E_{112}\\
&+q^{-5}(q^{-3}-q^3)E_{12}E_2E_{11122}
+[2]q^{-4}(q^{-3}-q^3)(q^2-1)^2E_{12}E_{12}^{(3)}\\
&-[2](q^{-1}-q)E_{12}^{(2)}E_2E_{112}
-q[2]^2(q^{-1}-q)^2E_{12}^{(2)}E_{12}^{(2)}\\
&=E_2^2E_{112}^{(2)}
+[2]q^{-7}(q^2-q+)(q^{-3}-q^3)E_2E_{12}E_{11122}-[2](q^{-1}-q)E_2E_{12}^{(2)}E_{
112}\\
&+[2]q^{-4}(q^{-3}-q^3)(q^2-1)^2E_{12}E_{12}^{(3)}
-[2](q^{-1}-q)E_{12}^{(2)}E_2E_{112}\\
&-q[2]^2(q^{-1}-q)^2E_{12}^{(2)}E_{12}^{(2)}\\
\Rightarrow\quad\left[E_{112}^{(2)},E_2^{(2)}\right]
&\stackrel{H}{=}0
  \end{align*}
These calculations become quicker, if summands in the integral form are
eliminated already during the calculation. 
 \end{enumerate}
\end{proof}

\newpage 
\section{Open Questions}\label{sec_question}

We finally give some open questions that the author would find interesting:\\

In view of the boundaries of the present article:
\begin{problem}\label{prob_uq}
  It would be desirable to have a short exact sequence for the full quantum
  group $U_q^\L(\g)$ instead of just the Borel part $U_q^\L(\g)^+$. 
\end{problem}

In view of our first Main Theorem \nref{thm_smalluq} on the structure of
$u_q^\L(\g)$ for small $q$:
\begin{problem}\label{prob_affine}
  One should calculate this table for affine quantum groups for $q$ of small
  order (compare
  the explicit Frobenius homomorphism for large order in \cite{CP97}). Depending
  on the case, one might expect a different affine root system of an infinite
  union of finite root systems. During the publication of this article, the author has indeed 
	calculated the respective Nichols algebras and root systems in \cite{Len14b}, but many 
	questions are open, in particular regarding the "`shifting"' of the isotropic roots.
\end{problem}

In view of our second Main Theorem \nref{thm_main} on the short exact sequence
for $U^\L_q(\g)$
\begin{problem}\label{prob_extension}
 The author has
already asked in  Oberwolfach (\cite{MFO14}  Question 5) for more examples or
even a classification of infinite-dimensional Hopf algebra extensions $H$
of a finite-dimensional pointed Hopf algebra $h$ by a universal enveloping
algebra $U$. There seem to be several sources of interesting examples:
  \begin{enumerate}[a)]
   \item The canonical examples is $H=U_q^\L(\g),h=u_q(\g),U=U(\g)$.
   \item By the results in this paper, for small roots of unity 
    $H=U_q^\L(\g),h=u_q(\g^{(0)}),U=U(\g^{(\ell)})$ are examples with
    $\g^{(0)}\neq\g\neq\g^{(\ell)}$.	
   \item The graded dual of Angiono's Pre-Nichols algebras \cite{An14} (which
    corresponds to a Kac-Procesi-DeConcini form). Here $U$ consists of 
    Cartan-type simple roots, but $h$ is a larger Nichols algebra. They
    conjecture several intriguing universal properties of $H$.
   \item The Hopf algebra in \cite{Good09} Construction 1.2., where
    (implicitly and again dual) $h$ is of type $A_1^{\times n}$ and $U$ is of
    type $A_1$. This should work much more general by joining suitable elements
    in $U(\g)$ for $U_q^\L(\g)$.
  \item The families of large-rank Nichols algebras $h$ over nonabelian groups
    constructed by the author in \cite{Len14a} using a diagram automorphism
    $\sigma$ on a Lie algebra $\g$ should yield examples with $U=U(\g)$ and $h$
    with root system $\g^\sigma$, e.g. $\g=E_6,\g^\sigma=F_4$.
  \end{enumerate}
  A good general classification approach should be to consider lifting data for
  Nichols algebras as in \cite{AS10}, but instead of coradical elements
  introduce new primitives (forming $U$) and then take the graded dual.\\
  Besides their theoretical charm, these extensions should have interesting
  applications to conformal field theory. E.g. the author's example b) for
  $\g=B_n,q=\pm i$ should correspond to $n$ symplectic fermions and the
  example d) is precisely the Hopf algebra considered by Gainutdinov,
  Tipunin for $W(p,p')$-models (unpublished).
\end{problem}

\end{document}